\documentclass[a4paper]{amsart}

\usepackage{amsfonts}
\usepackage{amssymb}
\usepackage{amsmath}
\usepackage{hyperref}
\usepackage{mathrsfs}
\usepackage{centernot}
\usepackage{mathdots}
\usepackage{stmaryrd}
\usepackage[all]{xy}

\newtheorem{thm}{Theorem}[section]
\newtheorem{lem}[thm]{Lemma}
\newtheorem{prp}[thm]{Proposition}
\newtheorem{cor}[thm]{Corollary}
\newtheorem{dfn}[thm]{Definition}

\newtheorem{que}{Question}

\newtheorem{baseexample}[thm]{Example} %never to be used !!! - use example environment
\newtheorem{baseremark}[thm]{Remark} %never to be used !!! - use remark environment

\newenvironment{example}
{\begin{baseexample}\rm}{\end{baseexample}}
\newenvironment{remark}
{\begin{baseremark}\rm}{\end{baseremark}}

\newcommand{\rem}[1]{}

\newcommand{\Step}[1]{\noindent {\bf Step #1.}} % For steps in proof

\newcommand{\N}{\mathbb{N}}
\newcommand{\Q}{\mathbb{Q}}
\newcommand{\R}{\mathbb{R}}
\newcommand{\Z}{\mathbb{Z}}

\newcommand{\calCapital}{
\newcommand{\calA}{{\mathcal{A}}}
\newcommand{\calB}{{\mathcal{B}}}
\newcommand{\calC}{{\mathcal{C}}}
\newcommand{\calD}{{\mathcal{D}}}
\newcommand{\calE}{{\mathcal{E}}}
\newcommand{\calF}{{\mathcal{F}}}
\newcommand{\calG}{{\mathcal{G}}}
\newcommand{\calH}{{\mathcal{H}}}
\newcommand{\calI}{{\mathcal{I}}}
\newcommand{\calJ}{{\mathcal{J}}}
\newcommand{\calK}{{\mathcal{K}}}
\newcommand{\calL}{{\mathcal{L}}}
\newcommand{\calM}{{\mathcal{M}}}
\newcommand{\calN}{{\mathcal{N}}}
\newcommand{\calO}{{\mathcal{O}}}
\newcommand{\calP}{{\mathcal{P}}}
\newcommand{\calQ}{{\mathcal{Q}}}
\newcommand{\calR}{{\mathcal{R}}}
\newcommand{\calS}{{\mathcal{S}}}
\newcommand{\calT}{{\mathcal{T}}}
\newcommand{\calU}{{\mathcal{U}}}
\newcommand{\calV}{{\mathcal{V}}}
\newcommand{\calW}{{\mathcal{W}}}
\newcommand{\calX}{{\mathcal{X}}}
\newcommand{\calY}{{\mathcal{Y}}}
\newcommand{\calZ}{{\mathcal{Z}}}
}

\newcommand{\bbCapital}{
\newcommand{\bbA}{{\mathbb{A}}}
\newcommand{\bbB}{{\mathbb{B}}}
\newcommand{\bbC}{{\mathbb{C}}}
\newcommand{\bbD}{{\mathbb{D}}}
\newcommand{\bbE}{{\mathbb{E}}}
\newcommand{\bbF}{{\mathbb{F}}}
\newcommand{\bbG}{{\mathbb{G}}}
\newcommand{\bbH}{{\mathbb{H}}}
\newcommand{\bbI}{{\mathbb{I}}}
\newcommand{\bbJ}{{\mathbb{J}}}
\newcommand{\bbK}{{\mathbb{K}}}
\newcommand{\bbL}{{\mathbb{L}}}
\newcommand{\bbM}{{\mathbb{M}}}
\newcommand{\bbN}{{\mathbb{N}}}
\newcommand{\bbO}{{\mathbb{O}}}
\newcommand{\bbP}{{\mathbb{P}}}
\newcommand{\bbQ}{{\mathbb{Q}}}
\newcommand{\bbR}{{\mathbb{R}}}
\newcommand{\bbS}{{\mathbb{S}}}
\newcommand{\bbT}{{\mathbb{T}}}
\newcommand{\bbU}{{\mathbb{U}}}
\newcommand{\bbV}{{\mathbb{V}}}
\newcommand{\bbW}{{\mathbb{W}}}
\newcommand{\bbX}{{\mathbb{X}}}
\newcommand{\bbY}{{\mathbb{Y}}}
\newcommand{\bbZ}{{\mathbb{Z}}}
}

\newcommand{\catCapital}{
\newcommand{\catA}{{\mathscr{A}}}
\newcommand{\catB}{{\mathscr{B}}}
\newcommand{\catC}{{\mathscr{C}}}
\newcommand{\catD}{{\mathscr{D}}}
\newcommand{\catE}{{\mathscr{E}}}
\newcommand{\catF}{{\mathscr{F}}}
\newcommand{\catG}{{\mathscr{G}}}
\newcommand{\catH}{{\mathscr{H}}}
\newcommand{\catI}{{\mathscr{I}}}
\newcommand{\catJ}{{\mathscr{J}}}
\newcommand{\catK}{{\mathscr{K}}}
\newcommand{\catL}{{\mathscr{L}}}
\newcommand{\catM}{{\mathscr{M}}}
\newcommand{\catN}{{\mathscr{N}}}
\newcommand{\catO}{{\mathscr{O}}}
\newcommand{\catP}{{\mathscr{P}}}
\newcommand{\catQ}{{\mathscr{Q}}}
\newcommand{\catR}{{\mathscr{R}}}
\newcommand{\catS}{{\mathscr{S}}}
\newcommand{\catT}{{\mathscr{T}}}
\newcommand{\catU}{{\mathscr{U}}}
\newcommand{\catV}{{\mathscr{V}}}
\newcommand{\catW}{{\mathscr{W}}}
\newcommand{\catX}{{\mathscr{X}}}
\newcommand{\catY}{{\mathscr{Y}}}
\newcommand{\catZ}{{\mathscr{Z}}}
}

\newcommand{\what}[1]{\widehat{#1}}

\newcommand{\veps}{\varepsilon}
\newcommand{\vphi}{\varphi}

\newcommand{\Lmb}{\Lambda}

\newcommand{\idealof}{\unlhd} % \vartriangleleft

\newcommand{\derives}{\Longrightarrow}
\newcommand{\nderives}{\centernot\Longrightarrow}

\newcommand{\suchthat}{\,:\,}
\newcommand{\where}{\,|\,}

\newcommand{\quo}[1]{\overline{#1}}

\newcommand{\SMatII}[4]{\left[\begin{array}{cc} {#1} & {#2} \\ {#3} &
{#4} \end{array}\right]}

\newcommand{\Circs}[1]{\left( #1 \right)}

 \DeclareMathOperator{\Aut}{Aut}
\DeclareMathOperator{\Br}{Br} %
\DeclareMathOperator{\Cent}{Cent} %
\DeclareMathOperator{\Char}{char} %
\DeclareMathOperator{\Cor}{Cor} %
\DeclareMathOperator{\End}{End} %
\DeclareMathOperator{\Hom}{Hom} %
\DeclareMathOperator{\id}{id} %
\DeclareMathOperator{\Inn}{Inn} %
\DeclareMathOperator{\Jac}{Jac} %

\newcommand{\op}{\mathrm{op}} %
\DeclareMathOperator{\Pic}{Pic} %
\DeclareMathOperator{\rank}{rank}
\DeclareMathOperator{\Res}{Res} %
\DeclareMathOperator{\Spec}{Spec} %
\newcommand{\nMat}[2]{\mathrm{M}_{#2}(#1)}

\newcommand{\dirlim}{\underrightarrow{\lim}\,}

\newcommand{\units}[1]{{#1^\times}}
\newcommand{\ideal}[1]{\left<#1\right>}

\newcommand{\rMod}[1]{{\mathrm{Mod}\textrm{-}{#1}}}

\newcommand{\rproj}[1]{{\mathrm{proj}}\textrm{-}{#1}}

\usepackage{longtable}
\usepackage[foot]{amsaddr}

\calCapital %
\bbCapital %
\catCapital %

\renewcommand{\Step}[1]{\medskip {\sc Step {#1}.}}

\newcommand{\lAd}[1]{\mathrm{Ad}^\ell_{#1}}
\newcommand{\rAd}[1]{\mathrm{Ad}^r_{#1}}

\newcommand{\mul}[1]{\odot_{#1}} %{\circledast_{#1}}

\newcommand{\Moreq}{\sim_\mathrm{Mor}}
\newcommand{\CMoreq}[1]{\sim_{\mathrm{Mor}/{#1}}}
\newcommand{\Breq}{\sim_\mathrm{Br}}
\newcommand{\nBreq}{\nsim_\mathrm{Br}}

\title{Rings That Are Morita Equivalent to Their Opposites}

\author{Uriya A.\ First$^*$}
\date{\today}
\address{$^*$Einstein Institute of Mathematics, Hebrew University of Jerusalem}
\email{uriya.first@gmail.com}
\thanks{This research was supported by an ERC grant \#226135 and by
the Lady Davis Fellowship Trust.}

\begin{document}

\maketitle

\begin{abstract}
    We consider the following problem: Under what assumptions are one or
    more of the following  equivalent for a ring $R$: (A) $R$ is Morita equivalent to a ring
    with involution, (B) $R$ is Morita equivalent to a ring with an anti-automorphism,
    (C) $R$ is Morita equivalent to its opposite ring.
    The problem is motivated by a theorem of Saltman which \emph{roughly} states that all conditions are equivalent for
    \emph{Azumaya algebras}.
    Based on the recent \emph{general bilinear forms} of \cite{Fi13A},
    we present a general machinery
    to attack the problem, and use it to show that
    (C)$\iff$(B)
    when $R$ is semilocal or \emph{$\Q$-finite}.
    Further results of similar flavor are also obtained, for example: If $R$ is a semilocal ring such that
    $\nMat{R}{n}$ has an involution, then $\nMat{R}{2}$ has an involution, and under further mild assumptions,
    $R$ itself has an involution.
    In contrast to that, we demonstrate that (B)$\nderives$(A).
    Our methods also give a new perspective on the
    Knus-Parimala-Srinivas proof of Saltman's Theorem.
    Finally, we give a method to test Azumaya algebras of exponent $2$ for the existence of involutions, and use
    it to construct explicit examples of such algebras.
\end{abstract}

\section{Overview}
\label{section:overview}

    Let $R$ be a ring. This paper is concerned with the  question of
    under what assumptions are all or some of the following conditions  equivalent:
    \begin{enumerate}
        \item[(A)] $R$ is Morita equivalent to a ring with involution,
        \item[(B)] $R$ is Morita equivalent to a ring with an anti-automorphism,
        \item[(C)] $R$ is Morita equivalent to $R^\op$ (the opposite ring of $R$).
    \end{enumerate}
    (Actually, we consider a slight refinement that takes into account the \emph{type} of
    the involution/anti-automorphism/Morita equivalence; see section \ref{section:Morita}.)
    Note that obviously (A)$\derives$(B)$\derives$(C), so one is interested in showing
    (B)$\derives$(A) or (C)$\derives$(B).

    The motivation for the question comes from  Azumaya algebras (we recall
    the  definition in section~\ref{section:azumaya}).
    Let $C$ be a commutative ring and let $A$ be an Azumaya $C$-algebra . It was shown
    by Saltman  \cite[Th.\ 3.1]{Sa78} that:
    \begin{enumerate}
        \item[(i)] $A$ is Brauer equivalent to an Azumaya algebra $B$ with an involution of the first kind
        if and only if $A$ is Brauer equivalent to $A^\op$.
        \item[(ii)] If $C/C_0$ is a Galois extension with Galois group $G=\{1,\sigma\}$ ($\sigma\neq 1$),
        then $A$ is Brauer equivalent to an Azumaya algebra $B$ with an involution whose restriction to $C$ is $\sigma$ if and only if
        the corestriction algebra $\Cor_{C/C_0}(A)=(A\otimes A^\sigma)^G$ is split (i.e.\ trivial in the Brauer group of $C_0$).
    \end{enumerate}
    In case $C$ is semilocal and connected, Saltman also showed that
    one can  take $B=A$ in (i) and (ii). (The case where $C$ is a field is an earlier
    result of Albert, e.g.\ see \cite[Ths.\ 10.19 \& 10.22]{Al61StructureOfAlgs}.)
    Two Azumaya algebras are Brauer equivalent if and only if they are Morita equivalent \emph{as $C$-algebras} (\cite[Cor.\ 17.2]{Ba64}),
    so (i) can be understood as: (C)$\derives$(A) for Azumaya algebras, provided the Morita
    equivalence is ``of the first kind''.
    A simpler proof of  Saltman's Theorem was later found by
    Knus, Parimala and Srinivas (\cite[\S4]{KnParSri90}).

    In this paper, we use \emph{general bilinear forms}, introduced in \cite{Fi13A},
    to give partial answer to Saltman's problem.
    More precisely, we show that the conditions (A) and (B) above can be phrased
    in terms of existence of certain bilinear forms, and use this observation to give some positive and negative results,
    Saltman's Theorem in particular.

\medskip

    We show that (C)$\derives$(B)
    when $R$ is semilocal or when $\dim_{\Q}(R\otimes_{\Z}\Q)$ and $|\ker(R\to R\otimes_{\Z}\Q)|$ are finite
    (i.e.\ when $R$ is \emph{$\Q$-finite}). Further results of similar flavor are also obtained.
    These include:
    \begin{enumerate}
        \item[(1)] If $R$ is \emph{semiperfect} (e.g.\ artinian)
        and Morita equivalent to $R^\op$, then $S$, the \emph{basic ring}
        that is Morita equivalent to $R$, has an anti-automorphism. In addition, if $R$ has an involution, then
        so does $\nMat{S}{2}$.
        \item[(2)] Suppose $R$ is semilocal. If $\nMat{R}{n}$ has an involution, then so does $\nMat{R}{2}$,
        and
        under mild assumptions (see Theorem~\ref{AZ:TH:involution-transfer} below), $R$ itself has an involution.
    \end{enumerate}
    In the special case that $R$ is a division ring, (2) implies that $\nMat{R}{n}$ has an involution if and only if
    $R$ has an involution, a result obtained by Albert (e.g.\  \cite[Th.\ 10.12]{Al61StructureOfAlgs}) when $[R:\Cent(R)]<\infty$
    and by Herstein  (e.g.\ \cite[Th.\ 1.2.2]{Her76}) in the general case.

    In contrast to the above, we show (B)$\nderives$(A) even when $R$ is  a finite dimensional
    algebra over a field,
    and even when it has an  anti-auto\-mor\-phi\-sm
    fixing the center pointwise. Whether (C)$\derives$(B) in  general  is still open (and we believe this is not the case).

\medskip

    We continue by describing
    the proof of Saltman's Theorem by Knus, Parimala and Srinivas  (\cite[\S4]{KnParSri90}) from the perspective of our methods.
    Namely, we  recover this proof as an application of our characterization of (A) in terms of
    general bilinear forms.
    This suppresses
    some of the computations of \cite[\S4]{KnParSri90}. We also recover the sharpening of Saltman's Theorem in case
    the base ring is semilocal.
    Finally, we apply our methods to construct non-trivial Azumaya algebras over Dedekind domains
    satisfying the conditions of Saltman's Theorem,
    but not admitting an involution. (However, these examples are not optimal; see section~\ref{section:no-inv}.)

\medskip

    The main problem of the paper was suggested to us by David Saltman himself after we gave a talk
    about \emph{general bilinear forms} at the 10th Brauer Group conference. The idea
    to apply the same methods for the construction of Azumaya algebras of exponent $2$ not admitting an involution
    was communicated to us by Asher Auel, after posting a previous version of this work on the internet.
    We deeply thank both of them for their contribution.

\medskip

    Sections~\ref{section:Morita} and~\ref{section:forms} are preliminaries:
    Section~\ref{section:Morita} recalls the basics of Morita theory,
    and
    section~\ref{section:forms} recalls general bilinear
    forms.
    In section~\ref{section:double-prog},
    we give a criterion
    in terms of bilinear forms  to when a ring is Morita equivalent to a ring with an involution (resp.\ anti-automorphism).
    This criterion is the core of this paper and it is used several times later.
    In sections~\ref{section:Mor-eq-to-op}--\ref{section:transfer},
    we show that (C)$\derives$(B) under certain finiteness assumptions, as well as other results of the same flavor.
    Section~\ref{section:examples} demonstrates that (B)$\nderives$(A).
    The rest of the paper concerns applications to Azumaya algebras:
    Section~\ref{section:azumaya} recalls some facts about Azumaya algebras, in
    section~\ref{section:first-kind} we show how our methods can reproduce the proof of Saltman's Theorem
    given in \cite{KnParSri90}, and in section~\ref{section:no-inv}, we construct Azumaya algebras of exponent
    $2$ without involution.
    Finally, section~\ref{section:questions} presents
    some questions that we were unable to answer.

\medskip

    {\bf Notation and conventions:} Unless specified otherwise, all rings are assumed to have a unity and ring
    homomorphisms are required to preserve it.
    Given a ring $R$, denote its set of invertible elements by $\units{R}$, its center by $\Cent(R)$,
    and its Jacobson radical by $\Jac(R)$. The $n\times n$
    matrices over $R$ are denoted by $\nMat{R}{n}$. The category of right $R$-modules is denoted by $\rMod{R}$
    and the category of f.g.\ projective right $R$-modules is denoted by $\rproj{R}$. For a subset $X\subseteq R$, we let
    $\Cent_R(X)$ denote the centralizer of $X$ in $R$.
    If a module $M$ can be viewed as a module over several rings, we use $M_R$ (resp.\ ${}_RM$) to denote ``$M$,
    considered as a right (resp.\ left) $R$-module''. Endomorphisms of left (right) modules are applied on the right (left).
    Throughout, a \emph{semisimple} ring means a \emph{semisimple artinian} ring.

\section{Morita Theory}
\label{section:Morita}

    In this section, we recall some facts from Morita Theory, and also
    give a refinement of the main problem of the paper. See \cite[\S18]{La99} or \cite[\S4.1]{Ro88} for proofs and
    further details.

\subsection{Morita Theory}

    Let $R$ be a ring. A right $R$-module $M$ is called a \emph{generator} if every right $R$-module
    is an epimorphic image of $\bigoplus_{i\in I}M$ for $I$ sufficiently large, or equivalently, if
    $R_R$ is a summand of $M^n$ for some $n\in\N$. The module $M$ is called a \emph{progenerator} if
    $M$ is a generator, finitely generated and projective. In this case, we also call $M$ a
    (right) $R$-progenerator.

    Let $S$ be another ring. An \emph{$(S,R)$-progenerator} is an $(S,R)$-bimodule $P$
    such that $P_R$ is a progenerator and $S=\End(P_R)$ (i.e.\ every endomorphism of $P_R$
    is of the form $p\mapsto sp$ for unique $s\in S$). In this case, ${}_SP$ is also a progenerator
    and $R=\End({}_SP)$.

    \smallskip

    The rings $R$, $S$ are said to be \emph{Morita equivalent}, denoted $R\Moreq S$, if the categories $\rMod{R}$
    and $\rMod{S}$ are equivalent. Morita's Theorems assert that:
    \begin{enumerate}
    \item For all equivalences\footnote{
        According to textbooks, an equivalence between two categories $\catA$ and $\catB$
        consists of a quartet $(F,G,\delta,\veps)$ such that $F:\catA\to \catB$
        and $G:\catB\to\catA$ are functors and $\delta:\id_{\catA}\to GF$ and $\veps:\id_{\catB}\to FG$ are natural isomorphisms.
        We do not need this detailed description here and hence we only specify $F$. In this case, the implicit functor
        $G$ is determined up to natural isomorphism.
    } $F:\rMod{S}\to \rMod{R}$
    there exists an $(S,R)$-pro\-genera\-tor $P$ such that $FM$ is \emph{naturally}
    isomorphic to $M\otimes_SP$ for all $M\in\rMod{S}$.
    \item Conversely, for any $(S,R)$-progenerator $P$ the functor $(-)\otimes_S P:\rMod{S}\to \rMod{R}$ is an equivalence of categories.
    \item There is a one-to-one correspondence between equivalences of categories $F:\rMod{S}\to \rMod{R}$
    (considered up to \emph{natural isomorphism}) and isomorphism classes of $(S,R)$-progenerators.
    The correspondence maps the composition of two equivalences to the tensor product of the corresponding progenerators.
    \end{enumerate}
    Note that (3) allows us to ``divide'' progenerators.
    For example, if $P$ is an $(S,R)$-progenerator and $Q$ is an $(S',R)$-progenerator,
    then there exists an $(S',S)$-pro\-genera\-tor $V$, unique up to isomorphism, such that $Q\cong V\otimes_S P$. Indeed,
    if $F:\rMod{S}\to\rMod{R}$ and $G:\rMod{S'}\to\rMod{R}$ correspond to $P$ and $Q$, respectively, then $V$ is the progenerator
    corresponding to $F^{-1}G:\rMod{S'}\to \rMod{S}$. (In fact,  one can take $V=\Hom_R(P,Q)$.)

\medskip

    Every $(S,R)$-progenerator $P$ induces an isomorphism
    $\sigma_P:\Cent(R)\to \Cent(S)$ given by $\sigma_P(r)=s$ where $s$ is the unique element of $\Cent(S)$
    satisfying $sp=pr$ for all $p\in P$. As $\sigma_P$  depends only on the isomorphism class of $P$, it follows
    that any equivalence of categories  $F:\rMod{S}\to\rMod{R}$ induces an isomorphism $\sigma_F:\Cent(R)\to \Cent(S)$.

    Let $C$ be a commutative ring and assume $R$ and $S$ are $C$-algebras. We say that $R$ and $S$ are Morita equivalent
    \emph{as $C$-algebras} or \emph{over $C$}, denoted $R\CMoreq{C} S$, if there exists  an equivalence $F:\rMod{S}\to \rMod{R}$
    such that $\sigma_F(c\cdot 1_R)=c\cdot 1_S$ for all $c\in C$. Equivalently, this means
    that there exists an $(S,R)$-progenerator $P$ such that $cp=pc$ for all $p\in P$ and $c\in C$.

    If $S$ is an arbitrary ring that is Morita equivalent to $R$ and $F:\rMod{S}\to\rMod{R}$ is any equivalence, then
    we can make $S$ into a $C$-algebra by letting $C$ act on $S$ via $\sigma_F$. In this setting, we have
    $S\CMoreq{C} R$.

\subsection{Types}
\label{subsection:types}

    To make the phrasing of some  results in the next sections easier, we now introduce \emph{types}:
    Let $C$ be a commutative ring and let $R$ and $S$ be  central $C$-algebras.
    (The algebra $R$ is central if the map $c\mapsto c\cdot 1_R:C\to\Cent(R)$
    is an isomorphism.)
    Every equivalence of categories $F:\rMod{S}\to \rMod{R}$ gives
    rise to an isomorphism $\sigma_F:\Cent(R)\to \Cent(S)$. As both $\Cent(R)$ and $\Cent(S)$
    are isomorphic to $C$, we can realize $\sigma_F$ as an automorphism of $C$, which
    we call the \emph{type} of $F$. (For example, when $F$ is of type $\id_C$, $R$ is Morita
    equivalent to $S$ as $C$-algebras.)
    Likewise, the type of an $(S,R)$-progenerator $P$ is the type of the equivalence induced by $P$.
    Namely, it is the unique automorphism $\sigma$ of $C$ satisfying $\sigma(c)p=pc$
    for all $p\in P$, $c\in C$.

    Let $\alpha$ be an automorphism or an anti-automorphism of $R$. The \emph{type} of $\alpha$ is defined to be its restriction to $C=\Cent(R)$.
    For example, an involution of $R$ is \emph{of the first kind} (i.e.\ it fixes $\Cent(R)$ pointwise) if and only if
    its type is $\id_{C}$.

    \medskip

    We now make an essential sharpening of the problem presented in section~\ref{section:overview}.
    Let $R$ be a ring, let $C=\Cent(R)$ and let $\sigma\in \Aut(C)$. We look for
    sufficient conditions ensuring that some or all of the following are equivalent:
    \begin{enumerate}
        \item[(A)] $R$ is Morita equivalent over $C$ to a (necessarily central) $C$-algebra with involution \emph{of type $\sigma$}.
        \item[(B)] $R$ is Morita equivalent over $C$ to a (necessarily central) $C$-algebra with an anti-automorphism \emph{of type $\sigma$}.
        \item[(C)] $R$ is Morita equivalent to $R^\op$ via equivalence of type $\sigma$ ($R^\op$ is considered
        as a $C$-algebra in the obvious way).
    \end{enumerate}
    Again, (A)$\derives$(B)$\derives$(C),\footnote{
        To see that (B)$\derives$(C), let $S$ be a central $C$-algebra admitting an anti-automorphism $\alpha$ of type $\sigma$,
        and assume that there is an $(S,R)$-progenerator $P$ of type $\id_C$ (i.e.\ $S\CMoreq{C} R$).
        Let $P'$ be the $(R^\op,S)$-bimodule obtained from $P$ by
        setting $r^\op \cdot p\cdot s=s^\alpha pr$. Then $P'$ is an $(R^\op,S)$-progenerator, hence
        $P'\otimes_S P$ is an $(R^\op,R)$-progenerator, and the latter is easily seen to have type $\sigma$.
    } so we want to show that (B)$\derives$(A) or (C)$\derives$(B).
    Satlman's Theorem (\cite[Th.\ 3.1]{Sa78}) for involutions of the first kind can now be phrased as (C)$\derives$(A) when $R/C$ is
    Azumaya and $\sigma=\id_C$.

\subsection{Progenerators and Scalar Extension}
\label{subsection:tensor-hom}

    We proceed by recalling several facts about the behavior of progenerators
    with respect to scalar extension.

    Throughout, $C$ is a commutative ring and $R$ is a $C$-algebra. All tensor
    products are taken over $C$. If $\sigma$ is
    an automorphism of $C$, then $R^\sigma$ denotes
    the $C$-algebra obtained from $R$ by letting $C$ act via $\sigma$. Observe
    that for all $M,N\in\rMod{R}$, $\Hom_R(M,N)$ admits a (right) $C$-module structure
    given by $(fc)m=(fm)c$ ($f\in\Hom_R(M,N)$, $c\in C$, $m\in M$).
    In addition, if $R'$ is another $C$-algebra and $M'\in\rMod{R}$, then $M\otimes M'$ has a right $R\otimes R'$-module
    structure given by $(m\otimes m')(r\otimes r')=mr\otimes m'r'$.
    We start with recalling  a well-known fact:

    \begin{prp}\label{AZ:PR:hom-of-tensor}
        Let $R'$ be a $C$-algebra and let $X,Y\in \rMod{R}$, $X',Y'\in\rMod{R'}$.
        Then there is a natural homomorphism of $C$-modules
        \[
            \Hom_R(X,Y)\otimes \Hom_{R'}(X',Y') ~\to~  \Hom_{R\otimes R'}(X\otimes X' ,Y\otimes Y')
        \]
        given by setting $(f\otimes f')(x\otimes x')=fx\otimes f'x'$. This homomorphism
        is an isomorphism whenever $X_R$, $X'_{R'}$ are finitely generated projective.
    \end{prp}

    \begin{proof}
        See \cite[pp.~14]{DeMeyIngr71SeparableAlgebras}, for instance. Alternatively, one can easily check
        that the map in the proposition is an isomorphism in case $X=R_R$ and $X'=R'_{R'}$.
        Since the map is natural (in the functorial sense), it is additive in $X$ and $X'$,
        hence it is an isomorphism whenever $X$ and $X'$ are summands of $R^n$ and $R'^m$, respectively.
    \end{proof}

    As a special case, we get:

    \begin{prp}\label{AZ:PR:standard-hom-III}
        Let $S$ be a $C$-algebra,
        let $R_S:=R\otimes S$, and set $X_S=X\otimes S$ for all $X\in\rMod{R}$.
        (Observe that $X_S$ is a right $R_S$-module via $(x\otimes s)(r\otimes s')=xr\otimes ss'$.)
        Then for all $X,Y\in\rMod{R}$,
        there is a natural homomorphism
        \[
        \Hom_R(X,Y)\otimes  S ~ \to ~ \Hom_{R_S}(X_S ,Y_S)
        \]
        given by setting $(f\otimes s)(x\otimes s')=fx\otimes ss'$ for
        $s,s'\in S$, $x\in X$, $r\in R$.
        This homomorphism is an isomorphism
        when $X_R$ is finitely generated projective.
    \end{prp}

    \begin{proof}
        Take $R'=X'=Y'=S$ in Proposition~\ref{AZ:PR:hom-of-tensor}.
    \end{proof}

    \begin{prp}\label{AZ:PR:scalar-ext-of-prog}
        Let $S$ and $D$ be $C$-algebras and let $P$ be an $(S, R)$-progenerator of type $\sigma\in\Aut(C)$.
        Put $R_D=R\otimes  D$, $S^\sigma_D=(S^\sigma)\otimes D$ and
        $P_D=P\otimes  D$, and endow $P_D$ with an
        $(S^\sigma_D,R_D)$-bimodule structure  by setting $(s\otimes d)(p\otimes d')(r\otimes d'')=(s pr)\otimes(dd'd'')$.
        Then
        $P_D$ is an $(S_D^\sigma,R_D)$-progenerator of type $\id_D$.
    \end{prp}

    \begin{proof}
        Since $P$ has type $\sigma$, $\End(P_R)\cong S^\sigma$ as $C$-algebras.
        By Proposition~\ref{AZ:PR:standard-hom-III}, $\End_{R_D}(P_D)=\Hom_{R_D}(P_D,P_D)\cong \Hom_R(P,P)\otimes D=S^\sigma\otimes D=S^\sigma_D$.
        It is routine to verify that the action of $S_D^\sigma$ on $P_D$ via endomorphisms is the action specified in the
        proposition.
        Finally, that $P$ is a right $R$-progenerator implies that $P_D$ is  a right $R_D$-progenerator,
        hence $P_D$ is an $(S^\sigma_D,R_D)$-progenerator. The type is verified by straightforward computation.
    \end{proof}

    \begin{prp}\label{AZ:PR:one-dual-commutes-with-scalar-ext-I}
        Let $P$ be an $(S,R)$-progenerator and let $N$ and $M$ be the
        prime radicals (resp.\ Jacobson radicals) of $R$ and $S$, respectively.
        View $\quo{P}:=P/PN$ as a right $\quo{R}:=R/N$-module.
        Then $PN=MP$, hence $\quo{P}=P/MP$
        admits a left $\quo{S}$-module structure. Furthermore,
        $\quo{P}$ is an $(\quo{S},\quo{R})$-progenerator.
    \end{prp}

    \begin{proof}
        By \cite[Pr.~18.44]{La99}, there is an isomorphism between the lattice of $R$-ideals and
        the lattice of $(S,R)$-submodules of $P$ given by $I\mapsto PI$. Similarly, the ideals of $S$
        correspond to $(S,R)$-submodules of $P$ via $J\mapsto JP$, hence every ideal $I\idealof R$
        admits a unique ideal $J\idealof S$ such that $JP=PI$. The ideal $J$ can also be
        described as $\Hom_R(P,PI)$. This description implies that
        $S/J= \Hom_R(P,P)/\Hom(P,PI)\cong\Hom_R(P,P/PI)\cong\End_{R/I}(P/PI)$. Thus,
        $P/PI$ is an $(S/J,R/I)$-progenerator.
        Choose $I=N$. Then by \cite[Cor.\ 18.45]{La99} (resp.\
        \cite[Cor.\ 18.50]{La99}) $J=M$, so we are done.
    \end{proof}

\section{General Bilinear Forms}
\label{section:forms}

    General bilinear forms were introduced in \cite{Fi13A}. In this section, we recall their basics
    and record several facts to be needed later. When not specified, proofs can be found in \cite[\S2]{Fi13A}.
    Throughout, $R$ is a (possibly non-commutative) ring.

    \begin{dfn} \label{AZ:DF:double-module}
        A \emph{(right) double $R$-module} is an additive group $K$ together with two
        operations $\mul{0},\mul{1}:K\times R\to K$ such that $K$ is a right $R$-module
        with respect to each of $\mul{0}$, $\mul{1}$ and
        \[(k\mul{0}a)\mul{1}b=(k\mul{1}b)\mul{0}a\qquad \forall ~k\in K,~a,b\in R\ .\]
        We let $K_i$ denote the $R$-module obtained by letting $R$ act on $K$ via $\mul{i}$.

        For two double $R$-modules $K,K'$, we define $\Hom(K,K')=\Hom_R(K_0,K'_0)\cap\Hom_R(K_1,K'_1)$.
        This makes the class of double $R$-modules into an abelian category (which is
        isomorphic to
        $\rMod{(R\otimes_\Z R)}$ and also to the category
        of $(R^\op,R)$-bimodules).\footnote{
            The usage of double $R$-modules, rather than $(R^\op,R)$-modules or $R\otimes_{\Z}R$-modules,
            was more convenient in \cite{Fi13A}, so we
            follow the notation of that paper. In addition, the notion of double module
            is more natural when considering bilinear forms as a special case
            of multilinear forms, where the form takes values in a (right) \emph{multi-$R$-module}.
        }

        An \emph{involution} of a double $R$-module $K$ is
        an additive bijective  map $\theta:K\to K$ satisfying  $\theta\circ \theta=\id_K$ and
        \[(k \mul{i} a)^\theta=k^\theta \mul{1-i} a \qquad\forall a\in R,~k\in K,~i\in\{0,1\}\ .\]
    \end{dfn}

    \begin{dfn}
    A (general) \emph{bilinear space} over $R$ is a triplet $(M,b,K)$ such that $M\in\rMod{R}$, $K$ is a double $R$-module
    and $b:M\times M\to K$ is a biadditive map
    satisfying
    \[b(xr,y)=b(x,y)\mul{0} r\qquad\textrm{and}\qquad b(x,yr)=b(x,y)\mul{1}r\]
    for all $x,y\in M$ and $r\in R$.
    In this case, $b$ is called a (general) \emph{bilinear form} (over $R$).
    Let $\theta$ be an involution of $K$.
    The form $b$ is called \emph{$\theta$-symmetric} if
    \[b(x,y)=b(y,x)^\theta\qquad\forall x,y\in M\ .\]
    \end{dfn}

    See \cite[\S2]{Fi13A} for various examples of general bilinear forms.

    \medskip

    Fix a double $R$-module $K$ and let $i\in\{0,1\}$. The \emph{$i$-$K$-dual} (or just $i$-dual) of
    an $R$-module $M$ is defined by
    \[M^{[i]}:=\Hom_{R}(M,K_{1-i})\ .\]
    Note that $M^{[i]}$  admits
    a right $R$-module structure given by $(fr)(m)=(fm)\mul{i}r$ (where $f\in M^{[i]}$, $r\in R$ and $m\in M$).
    In fact,
    $M\mapsto M^{[i]}$ is a left-exact contravariant functor from $\rMod{R}$ to itself, which we denote by $[i]$.
    Also observe that $R^{[i]}\cong K_i$ via $f\leftrightarrow f(1)$.

    Let $b:M\times M\to K$ be a (general) bilinear form. The \emph{left adjoint map} and \emph{right adjoint map}
    of $b$ are defined as follows:
    \[\lAd{b}:M\to M^{[0]},\quad (\lAd{b} x)(y)=b(x,y)\ ,\]
    \[\rAd{b}: M\to M^{[1]},\quad (\rAd{b} x)(y)=b(y,x)\ ,\]
    where $x,y\in M$. It is straightforward to check that $\lAd{b}$
    and $\rAd{b}$ are right $R$-linear. We say that
    $b$ is right (resp.\ left) \emph{regular} if $\rAd{b}$ (resp.\ $\lAd{b}$) is bijective.
    If $b$ is both right and left regular, we say that $b$ is regular. Left and right regularity are not
    equivalent properties; see  \cite[Ex.~2.6]{Fi13A}.

    Assume $b$ is regular. Then for every $w\in\End_R(M)$ there exists a \emph{unique} element $w^\alpha\in\End_R(M)$
    such that
    \[
    b(wx,y)=b(x,w^\alpha y)\qquad\forall x,y\in M\ .
    \]
    The map $w\mapsto w^\alpha$, denoted $\alpha$, turns out
    to be an anti-automorphism of $\End_R(M)$ which is called the (right) \emph{corresponding anti-automorphism} of $b$.
    (The left corresponding anti-automorphism of $b$ is the inverse of $\alpha$.)
    If $b$ is $\theta$-symmetric
    for some involution $\theta:K\to K$, then $\alpha$ is easily seen to be an involution.

    We say that two bilinear spaces $(M,b,K)$, $(M,b',K')$ are \emph{similar} if there is an isomorphism $f:K\to K'$
    such that $b'=f\circ b$. It is easy to
    see that in this case, $b$ and $b'$ have the same corresponding anti-automorphism, provided they are regular.

    \begin{thm}[{\cite[Th.\ 5.7]{Fi13A}}]\label{AZ:TH:correspondence}
        Let $M$ be a right $R$-generator. Then the map sending a regular
        bilinear form on $M$ to its corresponding anti-automorphism gives rise to a one-to-one correspondence between
        the class of regular bilinear forms on $M$, considered up to similarity, and the anti-automorphisms
        of $\End_R(M)$. Furthermore, the correspondence maps the equivalence
        classes of forms which are $\theta$-symmetric with respect to some $\theta$ to the involutions of $\End_R(M)$.
    \end{thm}

    The theorem implies that  every anti-automorphism $\alpha$ of $\End_R(M)$, with $M$ an $R$-generator, is induced by some regular bilinear
    form on $M$, which is unique up to similarity.
    We will denote this form by $b_\alpha$ and the double $R$-module
    in which it takes values by $K_\alpha$. In case $\alpha$ is an involution, then $b_\alpha$ is symmetric with
    respect to some involution of $K_\alpha$, which we denote by $\theta_\alpha$.

    The objects $K_\alpha$, $b_\alpha$ and $\theta_\alpha$ can be explicitly constructed as follows: Let $W=\End_R(M)$.
    Then $M$ is a left $W$-module. Using $\alpha$, we may view $M$ as a right $W$-module by
    defining $m\cdot w=w^\alpha m$. Denote by $M^\alpha$ the right $W$-module  obtained.
    Then $K_\alpha=M_\alpha\otimes_W M$. We make $K_\alpha$ into a double $R$-module by defining
    $(x\otimes y)\mul{0}r=x\otimes yr$ and $(x\otimes y)\mul{1} r=xr\otimes y$.
    The form $b_\alpha$ is given by $b_\alpha(x,y)=y\otimes x$, and when $\alpha$ is an involution,
    $\theta_\alpha:K_\alpha\to K_\alpha$ is defined by $(x\otimes y)^{\theta_\alpha}=y\otimes x$.

    \medskip

    We will also need the following technical proposition.

    \begin{prp}[{\cite[Lm.\ 7.7]{Fi13A}}]\label{AZ:PR:dfn-of-phi}
        Fix a double $R$-module $K$. For $M\in\rMod{R}$, define $\Phi_M:M\to M^{[1][0]}$ by
        $(\Phi_Mx)f=f(x)$ for all $x\in M$ and $f\in M^{[1]}$. Then:
        \begin{enumerate}
            \item[(i)] $\{\Phi_M\}_{M\in\rMod{R}}$ is a \emph{natural transformation} from
            $\id_{\rMod{R}}$ to $[0][1]$ (i.e.\ for all $N,N'\in\rMod{R}$ and $f\in \Hom_R(N,N')$, one has
            $f^{[1][0]}\circ\Phi_N=\Phi_{N'}\circ f$).
            \item[(ii)] $\Phi$ is additive (i.e.\ $\Phi_{N\oplus N'}=\Phi_N\oplus\Phi_{N'}$ for all $N,N'\in\rMod{R}$).
            \item[(iii)] $(\rAd{b})^{[0]}\circ\Phi_M=\lAd{b}$ for every general bilinear form $b:M\times M\to K$.
            \item[(iv)] $R^{[1][0]}$ can be identified with $\End_R(K_1)$. Under that identification,
            $(\Phi_Rr)k=k\mul{0} r$ for all $r\in R$ and $k\in K$.
        \end{enumerate}
    \end{prp}

    To finish, we recall that the orthogonal sum of two bilinear spaces $(M,b,K)$ and $(M',b',K)$ is defined
    to be $(M\oplus M',b\perp b',K)$  where $(b\perp b')(x\oplus x',y\oplus y')=b(x,y)+b'(x',y')$.
    The form $b\perp b'$ is right regular if and only if $b$ and $b'$ are right regular.

\section{Double Progenerators}
\label{section:double-prog}

    The observation which forms the basis of all results of this paper is the fact that whether
    a ring $R$ is equivalent to a ring with an anti-automorphism (resp.\ involution) can
    be phrased in terms of existence of certain bilinear forms (resp.\ double $R$-modules).
    In this section, we state and prove this criterion (Proposition~\ref{AZ:PR:being-mor-eq-to-ring-with-anti-auto}
    and Theorem~\ref{AZ:TH:being-mor-eq-to-ring-with-inv}).

    Throughout, $R$ is a ring and $C=\Cent(R)$. Recall from section~\ref{section:Morita}
    that for all $M,N\in\rMod{R}$, $\Hom_R(M,N)$ admits a (right) $C$-module structure. In particular,
    $\End_R(M)$ is a $C$-algebra.

    \medskip

    Let $K$ be a double $R$-module. Then $K$ can be viewed as
    an $(R^\op, R)$-bimodule by setting $a^\op\cdot k\cdot b=k\mul{0} b\mul{1}a$. If
    this bimodule is an $(R^\op,R)$-progenerator,
    we say that $K$ is a \emph{double $R$-progenerator}. The \emph{type} of a double $R$-progenerator
    is the type of its corresponding $(R^\op,R)$-module. Namely, it is the automorphism $\sigma\in\Aut(C)$
    satisfying $k\mul{0} c=k\mul{1}\sigma(c)$
    for all $c\in \Cent(R)$, $k\in K$.

    \smallskip

    It turns out that Theorem~\ref{AZ:TH:correspondence} is useful for producing double progenerators.

    \begin{lem}\label{AZ:LM:double-prog-factory}
        Let $M$ be an $R$-progenerator and let $\alpha$ be an anti-automorphism of
        $\End_R(M)$. Then $K_\alpha$ is a double $R$-progenerator. Viewing $\End_R(M)$
        as a $C$-algebra, the type of $K_\alpha$ is the type of $\alpha$.
    \end{lem}

    \begin{proof}
        Let $W=\End_R(M)$.
        Recall from section~\ref{section:forms} that $K_\alpha=M^\alpha\otimes_W M$. Since $M_R$ is an $R$-pro\-gene\-rator, $M$
        is a $(W,R)$-progenerator. Endow $M^\alpha$ with a left $R^\op$-module structure by
        setting $r^\op m=mr$. Then $M^\alpha$ is an $(R^\op,W)$-bimodule and, moreover, it is easily
        seen to be an $(R^\op,W)$-progenerator (because $M$ is a $(W,R)$-progenerator).
        By Morita theory (see section~\ref{section:Morita}),
        this means that $K_\alpha=M_\alpha\otimes_W M$ is an $(R^\op,R)$-progenerator, as required.
        That the type of $K_\alpha$ is the type of $\alpha$ follows by straightforward computation.
    \end{proof}

    \begin{lem}\label{AZ:LM:right-reg-iff-left-reg}
        Fix a double $R$-progenerator $K$. Then:
        \begin{enumerate}
            \item[(i)] $\Phi_M$ is an isomorphism for all $M\in\rproj{R}$ (see Proposition~\ref{AZ:PR:dfn-of-phi} for the
            definition of $\Phi$).
            \item[(ii)] The (contravariant) functors $[0],[1]:\rproj{R}\to\rproj{R}$ are mutual inverses.
            \item[(iii)] Let $(M,b,K)$ be a bilinear space with $M\in\rproj{R}$. Then $b$ is right regular
            if and only if $b$ is left regular.
        \end{enumerate}
    \end{lem}

    \begin{proof}
        (i) Since $K$ is a double $R$-progenerator,
        every automorphism of $K_1$ is of the form $k\mapsto k\mul{0}r$ for a unique $r\in R$.
        Hence, $\Phi_R$ is an isomorphism by Proposition~\ref{AZ:PR:dfn-of-phi}(iv). By the
        additivity of $\Phi$ (Proposition~\ref{AZ:PR:dfn-of-phi}(ii)), $\Phi_{R^n}$ is also an isomorphism for all $n$.
        As every $M\in\rproj{R}$ admits an $M'\in\rproj{R}$ with $M\oplus M'\cong R^n$, it follows that $\Phi_M$ is an isomorphism
        (since $\Phi_M\oplus\Phi_{M'}=\Phi_{R^n}$).

        (ii)
        Recall that $R^{[i]}\cong K_i\in\rproj{R}$. Since the functor $[i]$ is additive, it
        follows that $[i]$ takes $\rproj{R}$ into itself.
        By (i), $\Phi:\id_{\rproj{R}}\to [1][0]$ is a natural \emph{isomorphism}.
        For all $M\in\rMod{R}$, define $\Psi_M:M\to M^{[0][1]}$ by $(\Psi_M x)f=f(x)$, where $f\in M^{[0]}$ and $x\in M$.
        Then a similar argument  shows that $\Psi:\id_{\rproj{R}}\to [0][1]$ is a natural isomorphism.
        It follows that $[0]$ and $[1]$ are mutual inverses.

        (iii) By Proposition~\ref{AZ:PR:dfn-of-phi}(iii),
        $(\rAd{b})^{[0]}\circ\Phi_M=\lAd{b}$, so by (i), that $\rAd{b}$ is bijective implies $\lAd{b}$ is bijective.
        The other direction follows by symmetry. (Use the identity $(\lAd{b})^{[1]}\circ\Psi_M=\rAd{b}$.)
    \end{proof}

    \begin{prp}\label{AZ:PR:being-mor-eq-to-ring-with-anti-auto}
        The ring $R$ is Morita equivalent over $C$ to a (necessarily central) $C$-algebra with an anti-automorphism
        of type $\sigma\in\Aut(C)$ if and only if there exists a regular bilinear
        space $(M,b,K)$ such that $M$ is an $R$-progenerator and $K$ is a double $R$-progenerator
        of type $\sigma$.
    \end{prp}

    \begin{proof}
        If $(M,b,K)$ is a regular bilinear space,
        we may associated to
        $b$ its corresponding anti-automorphism
        $\alpha:\End_R(M)\to\End_R(M)$. Since $M$ is an $R$-progenerator, $\End_R(M)\CMoreq{C} R$ (recall that we view $\End_R(M)$
        as a $C$-algebra). In addition, for all $x,y\in M$ and $c\in C$,
        we have $b(x,y)\mul{1}\sigma(c)=b(x,y)\mul{0}c=b(cx,y)=b(x,c^\alpha y)=b(x,y)\mul{1}c^\alpha$, so since
        $K_1$ is faithful (because it is a progenerator), the type of $\alpha$
        is $\sigma$.

        Conversely, let $M$ be an $R$-progenerator and assume $\End_R(M)$ has an anti-automorphism of type $\sigma$.
        Then by Theorem~\ref{AZ:TH:correspondence}, $b_\alpha:M\times M\to K_\alpha$ is a regular
        bilinear form and $K_\alpha$ is a double $R$-progenerator of type  $\sigma$ by Lemma~\ref{AZ:LM:double-prog-factory}.
    \end{proof}

    \begin{lem}\label{AZ:LM:u-defin}
        Let $K$ be a double $R$-module with involution $\theta$.
        For every $M\in\rMod{R}$, define $u_{\theta,M}:M^{[0]}\to M^{[1]}$
        by $u_{\theta,M}(f)=\theta\circ f$. Then $u_{\theta,M}$ is a natural isomorphism of right $R$-modules.
    \end{lem}

    \begin{proof}
        By computation.
    \end{proof}

    The proof of following theorem demonstrates an idea that we wish to
    stress: One can construct involutions on
    rings that are Morita equivalent to $R$ by constructing symmetric general bilinear
    forms. Note that all possible involutions are obtained in this manner by Theorem~\ref{AZ:TH:correspondence}.
    (However, this
    fails if we limit ourselves to standard sesquilinear forms. Indeed, the base ring may not even
    have an anti-automorphism; see \cite[Exs.~2.4 \& 2.7]{Fi13A}, or section~\ref{section:no-inv}.)

    \begin{thm}\label{AZ:TH:being-mor-eq-to-ring-with-inv}
        The ring $R$ is Morita equivalent over $C$ to a $C$-algebra with an involution
        of type $\sigma\in\Aut(C)$ if and only if
        there exists a double $R$-progenerator of type $\sigma$ admitting an involution.
    \end{thm}

    \begin{proof}
        Assume $M_R$ is an $R$-progenerator such that $\End_R(M)$ has an involution of type
        $\sigma$. Then, as in the proof of Proposition~\ref{AZ:PR:being-mor-eq-to-ring-with-anti-auto},
        $K_\alpha$ is a double $R$-progenerator of type $\sigma$. Since $\alpha$ is an involution,
        $K_\alpha$ has an involution, namely $\theta_\alpha$.

        Conversely, assume $K$ is a double $R$-progenerator of type $\sigma$ with an involution $\theta$.
        Let $P$ be any right $R$-progenerator
        and let $M=P\oplus P^{[1]}$. Define $b:M\times M\to K$ by $b(x\oplus f, y\oplus g)=
        gx+(fy)^{\theta}$. The form
        $b$ is clearly $\theta$-symmetric.
        We claim that $b$ is regular. By Lemma~\ref{AZ:LM:right-reg-iff-left-reg}(iii) (or since
        $b$ is $\theta$-symmetric),
        it is enough to show only right regularity. Indeed, a straightforward computation shows that
        \[
        \rAd{b}=\SMatII{0}{\id_{P^{[1]}}}{u_{\theta,P^{[1]}}\circ\Phi_P}{0}\in \Hom_R(P\oplus P^{[1]},P^{[1]}\oplus P^{[1][1]})\ .
        \]
        As $\Phi_P$ and $u_{\theta,P^{[1]}}$
        are both bijective by Lemmas~\ref{AZ:LM:right-reg-iff-left-reg}(i) and~\ref{AZ:LM:u-defin}, respectively, so is $\rAd{b}$.
        Let $\alpha:\End_R(M)\to\End_R(M)$ be the anti-endomorphism
        associated with $b$. Then $\alpha$  is an involution (since $b$ is $\theta$-symmetric),
        and, as in the proof
        of Proposition~\ref{AZ:PR:being-mor-eq-to-ring-with-anti-auto},  $\alpha$ is of type $\sigma$.
    \end{proof}

    Using very similar ideas, one can show the following variation, which states when a particular
    ring $S$ that is Morita equivalent to $R$ over $C$ has an anti-automorphism or an involution of a given type.
    Compare with \cite[Th.\ 4.2]{Sa78}.

    \begin{prp}\label{AZ:PR:existence-of-anti-auto-crit}
        Let $M$ be a right $R$-progenerator. Then $\End_R(M)$ has an anti-automorphism (involution) if and only if
        there exists a double $R$-progenerator $K$ of type $\sigma$ (admitting an involution $\theta$)
        and a right regular ($\theta$-symmetric) bilinear form $b:M\times M\to K$.
    \end{prp}

\section{Rings That Are Morita Equivalent to Their Opposites}
\label{section:Mor-eq-to-op}

    Let $R$ be a ring.
    In this section, we use  Proposition~\ref{AZ:PR:being-mor-eq-to-ring-with-anti-auto} to show
    that under certain finiteness assumptions,
    $R\Moreq R^\op$ implies that $R$ is Morita equivalent to a ring with an anti-automorphism (i.e.\ (C)$\derives$(B)).
    Whether this holds in general is still open.
    Henceforth,
    we  freely consider  $(R^\op,R)$-progenerators as double $R$-progenerators and vice versa (see section~\ref{section:double-prog}).

    \medskip

    We begin with two lemmas whose purpose is to show that the functors $[0]$ and $[1]$
    commute with certain scalar extensions.

    \begin{lem}\label{AZ:LM:one-dual-commutes-with-scalar-ext-II}
        Let $K$ be an $(R^\op,R)$-progenerator of
        type $\sigma\in\Aut(\Cent(R))$, let $C\subseteq \Cent(R)$
        be a subring fixed pointwise by $\sigma$, and let $D/C$ be a commutative ring extension.
        For every $P\in\rproj{R}$, let
        $P_D:= P\otimes_C D\in\rproj{R_D}$.
        Make $K_D$ into an $(R_D^\op,R_D)$-module as in Proposition~\ref{AZ:PR:scalar-ext-of-prog}.
        Then $K_D$ is an $(R_D^\op,R_D)$-progenerator and $(P^{[1]})_D\cong (P_D)^{[1]}$
        for all $P\in \rproj{R}$ (the functor $[1]$ is computed with respect to $K$ in the left hand side and with respect
        to $K_D$ in the
        right hand side).
    \end{lem}

    \begin{proof}
        That $K_D$ is an $(R_D^\op,R_D)$-progenerator follows from Proposition~\ref{AZ:PR:scalar-ext-of-prog} (observe
        that $(R^\op)_D^\sigma=(R^\op)_D=(R_D)^\op$ since $\sigma$ fixes $C$). That $(P^{[1]})_D\cong (P_D)^{[1]}$
        follows from Proposition~\ref{AZ:PR:standard-hom-III}.
    \end{proof}

    \begin{lem}\label{AZ:LM:one-dual-commutes-with-scalar-ext-I}
        Let $K$ be an $(R^\op,R)$-progenerator and let $N$ be the
        prime (resp.\ Jacobson) radical  of $R$. For all $P\in\rproj{R}$, define
        $\quo{P}:=P/PN\cong P\otimes_R (R/N)\in\rproj{R/N}$.
        Then $\End_{\quo{R}}(\quo{K})\cong\quo{R}^\op$,
        $\quo{K}$ is an $(\quo{R}^\op,\quo{R})$-progenerator and $\quo{P^{[1]}}\cong \quo{P}^{[1]}$
        (the functor $[1]$ is computed with respect to $K$ in the left hand side and with respect to $\quo{K}$ in the
        right hand side).
    \end{lem}

    \begin{proof}
        By Proposition~\ref{AZ:PR:one-dual-commutes-with-scalar-ext-I}, $KN=N^\op K$ and $\quo{K}$
        is an $(\quo{R}^\op,\quo{R})$-progenerator.
        Let $P\in\rproj{R}$.
        We claim that $P^{[1]}N=\Hom_R(P,(KN)_0)$ (we consider $KN$ as a double $R$-module here).
        Indeed, by definition, $P^{[1]}N=\Hom_R(P,K_0)N\subseteq \Hom_R(P,(K\mul{1} N)_0)=\Hom_R(P,(N^\op K)_0)=\Hom_R(P,(KN)_0)$,
        so there is a natural inclusion $P^{[1]}N\subseteq \Hom_R(P,(KN)_0)$. As this inclusion is additive in $P$,
        it is enough to verify the equality for $P=R_R$, which is routine. Now, we get
        \[
        \quo{P^{[1]}}=\frac{\Hom_R(P,K_0)}{P^{[1]}N}=
        \frac{\Hom_R(P,K_0)}{\Hom_R(P,(KN)_0)}\cong \Hom_R(P,\quo{K}_0)\cong \Hom_{\quo{R}}(\quo{P},\quo{K}_0)=\quo{P}^{[1]},
        \]
        as required.
    \end{proof}

    Recall that a ring $R$ is called \emph{$\Q$-finite} if
    both $\dim_{\Q}(R\otimes_{\Z}\Q)$ and the cardinality of $\ker(R\to R\otimes_{\Z}\Q)$ are finite.
    The following lemma   is inspired by \cite[Pr.\ 18.2]{Ba64}.

    \begin{lem}\label{AZ:LM:Q-finite-rings}
        Assume $R$ is $\Q$-finite and let $N$ be the prime radical of $R$. Then $N$ is nilpotent
        and $R/N\cong T\times \Lmb$ where $T$ is a semisimple finite ring
        and  $\Lmb$ is a subring of a semisimple $\Q$-algebra $E$ such that $\Lmb\Q=E$.
    \end{lem}

    \begin{proof}
        Let $T=\ker(R\to R\otimes_{\Z}\Q)$. Then $T$ is an ideal of $R$.
        Consider $T$ as a non-unital ring and let $J=\Jac(T)$ (i.e.\ $J$ is the intersection
        of the kernels of all right $T$-module morphisms
        $T\to M$ where $M$ is a simple right $T$-module  with $MT=M$). Arguing as in \cite[Pr.\ 18.2]{Ba64},
        we see that $J$ is also an $R$-ideal. Since $J$ and all its submodules are finite (as sets), $J^n=0$ for some $n$
        (because $MJ\subseteq N$ for any right $T$-module $M$ and any  maximal submodule $N\leq M$).
        In particular, $J\subseteq N$.

        Replacing $R$ with $R/J$, we may assume $J=0$.
        Now, $T$ is semisimple and of finite length, hence it has a unit $e$. As $er,re\in T$ for all $r\in R$,
        we see that $er=ere=re$. Thus, $e\in\Cent(R)$ and $R\cong T\times (1-e)R$. As $\Lmb:=(1-e)R$ is torsion-free,
        it is a subring of $E:=\Lmb\otimes_{\Z}\Q$, which is a f.d.\ $\Q$-algebra by assumption.

        Let $I=\Jac(E)\cap\Lmb$.
        Then $I$ is nilpotent, hence $0\times I\subseteq N$.
        Replacing $R$ by $R/(0\times I)$, we may assume $E$ semisimple.
        We are thus finished if we show that the prime radical of $\Lmb$, denoted $N'$, is $0$ (because then $N=0\times N'=0$).
        Indeed, by \cite[Th.\ 2.5]{BayKeaWil87}, $\Lmb$ is noetherian (here we need $E$ to be semisimple).
        Thus, $N'$ is nil, hence so is $N'\Q\idealof E$. But $E$ is semisimple, so we must have $N'\subseteq N'\Q=0$.
    \end{proof}

    \begin{thm}\label{AZ:TH:R-mor-eq-to-R-op}
        Assume $R$ is semilocal or $\Q$-finite and let $K$ be an $(R^\op,R)$-progenerator.
        Then for every $P\in\rproj{R}$ there is $n\in\N$ such that $P\cong P^{ [1]^n}:=P^{[1][1]\dots[1]}$ ($n$ times).
        When $R$ is semilocal, $n$ is independent of $P$.
    \end{thm}

    \begin{proof}
        Let $N$ denote the Jacobson radical of $R$ in case $R$ is semilocal
        and the prime radical of $R$ otherwise. In the latter case, $N$
        is nilpotent by Lemma~\ref{AZ:LM:Q-finite-rings}, so $N\subseteq\Jac(R)$ in both cases.

        Using the notation of Lemma~\ref{AZ:LM:one-dual-commutes-with-scalar-ext-I}, observe
        that every $P\in\rproj{R}$ is the projective cover of $\quo{P}=P/PN$.
        As projective covers are
        unique up to isomorphism, we have $P\cong Q$ $\iff$ $\quo{P}\cong \quo{Q}$ for all $P,Q\in\rproj{R}$.
        Therefore, using Lemma~\ref{AZ:LM:one-dual-commutes-with-scalar-ext-I}, we may assume $N=0$. Thus,
        $R$ is semisimple or $R=T\times \Lmb$ with $\Lmb\otimes_\Z\Q$ semisimple  and $T$
        finite and semisimple.

        Assume $R$ is semisimple and let $V_1,\dots,V_t$ be a complete
        set of simple $R$-modules up to isomorphism. By Lemma~\ref{AZ:LM:right-reg-iff-left-reg}(ii), $[1]$ permutes isomorphism
        classes of f.g.\ projective $R$-modules,
        and since $[1]$ is additive,  it permutes the indecomposable projective modules,
        namely, $V_1,\dots,V_t$. Therefore, there
        is $n\in\N$ (say, $n=t!$) such that $V_i^{[1]^n}\cong V_i$ for all $i$. As any $P\in\rproj{R}$
        is a direct sum of simple modules, we get $P^{[1]^n}\cong P$. We also record that $\mathrm{length}(P^{[1]})=\mathrm{length}(P)$.

        Now assume $R=T\times \Lmb$  as above, and set $E:=\Lmb\otimes_{\Z}\Q$ and $S:=T\times E$.
        Let $\sigma\in\Aut(\Cent(R))$ be the type of $K$ and let
        $e=(1_T,0)\in R$. Then $e$ is the maximal torsion idempotent in $\Cent(R)$, hence $\sigma(e)=e$.
        Define $C:=\Z e+\Z(1-e)$ and $D:=\Z e+\Q(1-e)+\subseteq T\times E=S$. Then $\sigma|_C=\id_C$
        and $S\cong R\otimes_C D$. Let $K'=K\otimes_CD$. Then by
        Lemma~\ref{AZ:LM:one-dual-commutes-with-scalar-ext-II}, $K'$ is an $(S^\op,S)$-progenerator
        and $(P_S)^{[1]}\cong (P^{[1]})_S$ for all $P\in\rproj{R}$ (where $P_S:=P\otimes_R S$).

        For every $P\in\rproj{R}$, define $j(P)=\mathrm{length}(P_S)$. Since $S$ is semisimple, the previous paragraphs
        imply that $j(P^{[1]})=j(P)$.
        Therefore, we are done if we show that for all $m\in\N$ there are finitely many isomorphism classes of modules $P$
        with $j(P)=m$ (for this  implies that action of $[1]$ have finite orbits on $\rproj{R}/\!\cong$, which yields the theorem).
        Indeed, thanks to \cite[Th.\ 2.8]{BayKeaWil87}, up to isomorphism,
        there are finitely many $\Lambda$-modules of any given $\Z$-rank, and this is easily seen to imply that there are
        finitely many isomorphism classes of modules $P\in\rproj{R}$ with $P_S$ of a given length, as required.
    \end{proof}

    \begin{remark}
        It is also possible to prove the semilocal case of
        Theorem~\ref{AZ:TH:R-mor-eq-to-R-op} by using a result of Facchini and Herbera \cite[Cor.\ 2.13]{FacHer00} stating that
        every semilocal ring $R$ has only finitely many indecomposable projective modules up to isomorphism,
        and every f.g.\ projective module is a direct sum of these modules.
    \end{remark}

    \begin{cor}\label{AZ:CR:three-implies-two}
        Let $C=\Cent(R)$, $\sigma\in\Aut(C)$ and
        assume $R$ is semilocal or $\Q$-finite.
        Then $R$ is Morita equivalent over $C$ to a (central) $C$-algebra with an anti-automorphism of type $\sigma$
        if and only if $R$ is Morita equivalent to $R^\op$ (as rings) via Morita equivalence of type $\sigma$.
    \end{cor}

    \begin{proof}
        We only check the nontrivial direction.
        Assume $R$ is Morita equivalent to $R^\op$  via Morita equivalence of type $\sigma$.
        Then there exists a double $R$-progenerator $K$ of type $\sigma$. By Theorem~\ref{AZ:TH:R-mor-eq-to-R-op},
        there exists $n\in\N$ such that $R^{[1]^n}\cong R_R$. Let $M=\bigoplus_{m=0}^{n-1}R^{[1]^m}$. Then
        there is an isomorphism $f:M\to M^{[1]}$. This isomorphism gives rise to a right regular bilinear space
        $(M,b,K)$ with $f=\rAd{b}$, namely, $b:M\times M\to K$ is given by $b(x,y)=(fy)x$.
        By Lemma~\ref{AZ:LM:right-reg-iff-left-reg}(iii), $b$ is regular,
        so we are  done by Proposition~\ref{AZ:PR:being-mor-eq-to-ring-with-anti-auto}.
    \end{proof}

    The proof of Corollary \ref{AZ:CR:three-implies-two} cannot be applied
    to arbitrary rings since there are double $R$-modules $K$ for which $M^{[1]}\ncong M$
    for all $0\neq M\in\rproj{R}$:

    \begin{example}\label{AZ:EX:bad-progenerator}
        Let $F$ be a field and let $R=\dirlim\{\nMat{F}{2}^{\otimes n}\}_{n\in\N}$.
        Then any f.g.\ projective right $R$-module is obtained by scalar extension from
        a f.g.\ projective module over $\nMat{F}{2}^{\otimes n}\hookrightarrow R$. Using this, it
        not
        hard (but tedious) to show that the monoid $(\rproj{R}/\!\cong,\oplus)$
        is isomorphic to $(\Z[\frac{1}{2}]\cap [0,\infty),+)$. If $V_n$ is the unique simple
        projective right module over $\nMat{F}{2}^{\otimes n}$, then $V_n\otimes R$ is mapped to $2^{-n}$.

        Let $T$ denote the transpose involution on $\nMat{F}{2}$. Then
        $\what{T}=\dirlim\{T^{\otimes n}\}_{n\in\N}$ is an involution of $R$.
        Let $K=R^2\in\rproj{R}$. Then $\End_R(K)\cong \nMat{R}{2}\cong R$ and using
        $\what{T}$, we can identify $\End_R(K)$ with $R^\op$, thus making $P$ into an $(R^\op,R)$-progenerator.
        We claim that for $M\in\rproj{R}$ and $n\in\N$, $M^{[1]}\cong M$ implies $M=0$.
        To see this,
        let $\vphi_1$ be the map obtained from $[1]$ by identifying $\rproj{R}/\!\cong$
        with $\Z[\frac{1}{2}]\cap [0,\infty)$.
        Then $\vphi_1(2)=1$ because $(R_R^2)^{[1]}\cong K_1^2\cong R_R$.
        Therefore, $\vphi_1(x)=\frac{1}{2}x$
        for all $x\in \Z[\frac{1}{2}]\cap [0,\infty)$, which means
        that $\vphi_1(x)\neq x$ for all $0\neq x\in \Z[\frac{1}{2}]\cap [0,\infty)$.
    \end{example}

\section{Semiperfect Rings}
\label{section:semiperfect}

    Let $R$ be a semilocal ring that is Morita equivalent to its opposite ring. While Corollary~\ref{AZ:CR:three-implies-two}
    implies that $R$ is Morita equivalent to  a ring with an anti-automorphism,  it does not provide any information
    about this ring. However, when  $R$ is \emph{semiperfect}, we can
    actually point out a specific ring which is Morita equivalent to $R$ and has an anti-automorphism.

    Recall that a ring $R$ is semiperfect if $R$ is semilocal
    and idempotents lift modulo $\Jac(R)$ (e.g.\ if $\Jac(R)$ is nil). In this case, the
    map
    \[P\mapsto P/P\Jac(R)~:~\rproj{R}/\!\cong\quad\to\quad \rproj{(R/\Jac(R))}/\!\cong\]
    is bijective (e.g.\ see \cite[\S2.9]{Ro88} or
    \cite[Th.\ 2.1]{Ba60}). Thus, up to isomorphism,
    $R$ admits finitely many indecomposable projective right $R$-modules
    $P_1,\dots,P_t$ and any $P\in\rproj{R}$ can be written as $P\cong \bigoplus_{i=1}^t P_i^{n_i}$
    with $n_1,\dots,n_t$ uniquely determined. In particular, $R_R\cong \bigoplus_{i=1}^t P_i^{m_i}$
    for some (necessarily positive) $m_1,\dots,m_t$.
    The ring $R$ is called \emph{basic} if $m_1=\dots=m_t=1$, namely, if $R_R$ is a sum of
    non-isomorphic indecomposable projective modules. 
    It is well-known that every semiperfect ring $R$
    admits a basic ring $S$ that is Morita equivalent to it over $\Cent(R)$, and $S$ is well-determined
    up to isomorphism as a $\Cent(R)$-algebra. (Indeed, $S$ and the map $\Cent(R)\to S$ can be determined from $\rMod{R}$;
    see
    \cite[Prp.\ 18.37]{La99}, the preceding discussion, and \cite[Rm.\ 18.43]{La99}. Once $S$ is fixed,
    the Morita equivalence  $\rMod{S}\to\rMod{R}$ is not unique in general.)
    Explicitly, we can take $S=\End_R(M)$, where $M=P_1\oplus\dots\oplus P_t$.
    (For example,
    if $R=\nMat{L}{n}$ with $L$ a local ring, then $S=L$.)

    Assume now that there is an $(R^\op,R)$-progenerator $K$ of type $\sigma\in\Aut(\Cent(R))$. Then the functor $[1]$ must
    permute the isomorphism classes of $P_1,\dots,P_t$ (because they are the only indecomposable  modules in
    $\rproj{R}$) and hence stabilize $M=P_1\oplus\dots\oplus P_t$. Therefore, as in the proof of Corollary~\ref{AZ:CR:three-implies-two},
    $\End_R(M)$, the basic ring which is Morita equivalent to $R$,
    has an anti-automorphism of type $\sigma$. We have thus obtained the following proposition.

    \begin{prp}\label{AZ:PR:anti-automorphisms-pass-to-the-basic-ring}
        Let $R$ be a semiperfect ring and let
        $S$ be a basic ring that is Morita equivalent to $R$ over $\Cent(R)$.
        Then $R$ is Morita equivalent to $R^\op$ via
        an equivalence of type $\sigma\in\Aut(\Cent(R))$
        if and only if $S$ has an anti-automorphism of type $\sigma$.
    \end{prp}

    Proposition \ref{AZ:PR:anti-automorphisms-pass-to-the-basic-ring}
    has a version for involutions in which the ring $S$ is replaced
    with $\nMat{S}{2}$.

    \begin{prp}\label{AZ:PR:involutions-pass-to-the-basic-ring}
        Let $R$ be a semiperfect ring and let
        $S$ be a basic ring that is Morita equivalent to $R$ over $\Cent(R)$.
        If $R$ has an involution of type $\sigma$, then so does $\nMat{S}{2}$.
    \end{prp}

    \begin{proof}
        Let $\alpha$ be an involution of $R$ of type $\sigma$, and let $K$ be the double $R$-module obtained
        from $R$ by setting $k\mul{0}r=r^\alpha k$ and $k\mul{1}r=kr$ ($k,r\in R$).
        Then $K$ is double $R$-progenerator of type $\sigma$ admitting an involution $\theta:=\alpha$.
        For any $P\in\rproj{R}$, let $b_P$ denote the bilinear form $b$ constructed in
        the proof of Theorem~\ref{AZ:TH:being-mor-eq-to-ring-with-inv}. Then
        $(b_P,P\oplus P^{[1]},K)$ is an $\alpha$-symmetric  regular bilinear space.
        Let $P_1,\dots,P_t$ be a complete list of indecomposable projective right $R$-modules up to isomorphism.
        Then $b:=b_{P_1}\perp\dots\perp b_{P_t}$ is a  regular $\alpha$-symmetric bilinear form
        defined over $M\oplus M^{[1]}$, where $M=P_1\oplus\dots\oplus P_t$. Therefore,
        $\End_R(M\oplus M^{[1]})$ has an involution of type $\sigma$, namely, the corresponding anti-automorphism
        of $b$. However,
        we have seen above that
        $M\cong M^{[1]}$, so $\End_R(M\oplus M^{[1]})\cong\nMat{\End_R(M)}{2}\cong\nMat{S}{2}$ as $\Cent(R)$-algebras.
    \end{proof}

    In fact, in many cases, the assumptions of Proposition~\ref{AZ:PR:involutions-pass-to-the-basic-ring}
    imply that $S$ itself has an involution. The general statement and its proof will be given in the next section.
    However, in case $S$ is a division ring or local with with $2\in\units{S}$,
    there is a significantly simpler proof, with which we finish this section.

    \begin{prp}\label{AZ:PR:inv-over-local-rings}
        Let $L$ be a local ring
        and let $R=\nMat{L}{n}$. Assume $2\in\units{L}$ or $L$ is a division ring. Then
        $R$ has an involution of type $\sigma$ if and only if $L$ has an involution of type $\sigma$.
    \end{prp}

    \begin{proof}
        That $R$ has an involution when $L$ has an involution is obvious.
        Conversely, let $\alpha$ be an involution of $R$ and let $K$, $\theta$ be as in the proof of
        Proposition~\ref{AZ:PR:involutions-pass-to-the-basic-ring}.
        Let $P$ be the unique indecomposable projective right $R$-module.
        Then necessarily $P\cong P^{[1]}$. Fix an isomorphism $f:P\to P^{[1]}$
        and observe that the bilinear form $b(x,y):=(fy)x+((fx)y)^\theta$
        (resp.\ $b'(x,y):=(fy)x-((fx)y)^\theta$) is $\theta$-symmetric (resp.\ $(-\theta)$-symmetric).
        In addition, $\rAd{b}+\rAd{b'}=2f$.

        Assume $2\in\units{L}$. Then $f^{-1}\circ\rAd{b}+f^{-1}\circ\rAd{b'}=2\id_P$.
        Since the r.h.s.\ is invertible and lies in the local ring $\End_R(P)\cong L$, one of  $f^{-1}\circ\rAd{b}$, $f^{-1}\circ\rAd{b'}$
        must be invertible, hence one of $b$, $b'$ is right regular. In any case, we get that $L\cong\End_R(P)$ has an involution
        of type $\sigma$, namely,  the corresponding anti-automorphism of $b$ or $b'$.

        When $L$ is a division ring, $P$ is simple, so $b'$ is regular if $\rAd{b'}\neq 0$.
        If $\rAd{b'}=0$, then $b'=0$, hence $(fy)x=((fx)y)^\theta$ for all $x,y\in P$.
        This means that the bilinear form $b''(x,y):=(fy)x$ is $\theta$-symmetric.
        As $\rAd{b''}=f$, $b''$ is regular.
    \end{proof}

    \begin{remark}
        In case $L$ is a division ring, Proposition~\ref{AZ:PR:inv-over-local-rings} follows
        from \cite[Th.~1.2.2]{Her76}. The case where $L$ is also finite dimensional
        over its center was noted earlier by Albert (e.g.\ see \cite[Th.~10.12]{Al61StructureOfAlgs}).
    \end{remark}

\section{Transferring Involutions}
\label{section:transfer}

    Motivated by Proposition~\ref{AZ:PR:involutions-pass-to-the-basic-ring},
    this section concerns the question of whether the fact that the ring $\nMat{R}{n}$
    has an involution implies that $R$ has an involution. We shall provide
    a positive answer for ``most'' semilocal rings, but the question is still open for general semilocal rings,
    and even for finite dimensional algebras over fields.
    For non-semilocal rings, that $\nMat{R}{n}$ has an involution does not imply that $R$ has an anti-automorphism; an example
    will be given in section~\ref{section:no-inv}.

    Throughout, $R$ is a ring, $C=\Cent(R)$, $\sigma\in\Aut(C)$, and $n\in\N$ is fixed.

    \medskip

    We begin by introducing some notation. For an anti-endomorphism $\gamma:R\to R$,
    we define the \emph{standard double $R$-module} of $(R,\gamma)$ to be $R$
    endowed with the actions $k\mul{0}r=r^\gamma k$ and $k\mul{1}r=kr$. Observe that
    when $\gamma$ is invertible,
    this yields a double $R$-progenerator of same  type as $\gamma$. We now recall the following
    theorem, which follows from Th.~7.8 in \cite{Fi13A} and the comment following the statement.
    To make the exposition more self-contained, we present here an ad-hoc proof which is slightly simpler
    than the proof given in \cite{Fi13A}, but not essentially different.

    \begin{thm}\label{AZ:TH:anti-autos-are-transposes}
        Assume that up to isomorphism $M=R_R$ is the only right $R$-module
        satisfying $M^n\cong R^n$ (e.g.\ when $R$ is semilocal).
        Then for every anti-automorphism $\alpha:\nMat{R}{n}\to\nMat{R}{n}$,
        there exists an anti-automorphism $\gamma:R\to R$
        such that $K_\alpha$ is isomorphic to the standard double $R$-module
        of $(R,\gamma)$. Furthermore, there is an inner automorphism
        $\vphi\in\Inn(\nMat{R}{n})$ such that $\vphi\circ \alpha=T_n\gamma$
        where $T_n\gamma$ is defined by $(r_{ij})^{T_n\gamma}=(r_{ji}^\gamma)$.
    \end{thm}

    \begin{proof}
        Identify $\nMat{R}{n}$ with $\End_R(R^n)$, and let $b_\alpha:R^n\times R^n\to K:=K_\alpha$
        be the bilinear form induced by $\alpha$.
        Since $b_\alpha$ is regular, we have an isomorphism $\rAd{b_\alpha}:R^n\to (R^n)^{[1]}\cong(R^{[1]})^n$.
        By assumption, this means that $R_R \cong R^{[1]}$. Observe that $R^{[1]}=\Hom_R(R,K_0)\cong K_1$
        via $f\mapsto f(1)$. Thus, we may identify $R_R\cong R^{[1]}$ with $K_1$, so now $k\mul{1}r=kr$ for all $k,r\in K=R$.
        Define $\gamma:R\to R$ by $r^\gamma=1_R\mul{0} r$, and note that
        $k\mul{0}r=1\mul{0}r\mul{1}k=r^\gamma \mul{1}k=r^\gamma k$.
        This easily implies that $\gamma$ is an anti-\emph{endo}morphism
        of $R$, and that $K$ is just the standard double $R$-module of $(R,\gamma)$.
        Since $K$ is a double $R$-progenerator, we must have $R^\op\cong \End(K_1)$
        where the isomorphism is given by $r^\op\mapsto [k\mapsto k\mul{0}r]$. As
        $k\mul{0}r=r^\gamma k$, this forces $\gamma$ to be an anti-\emph{auto}morphism.
        Finally, we refer the reader to the proof of \cite[Th.~7.8(iv)]{Fi13A} for the existence of $\vphi$.
    \end{proof}

    As a corollary, we get the following theorem, which can be regarded as a
    variation of Propositions~\ref{AZ:PR:involutions-pass-to-the-basic-ring}.

    \begin{thm}\label{AZ:TH:involution-transfer-weak}
        Suppose that up to isomorphism $M=R_R$ is the only right $R$-module
        satisfying $M^n\cong R^n$. If $\nMat{R}{n}$ has an involution (resp.\ anti-automorphism) of type $\sigma$,
        then $\nMat{R}{2}$ (resp.\ $R$) has an involution (resp.\ anti-automorphism) of type $\sigma$.
    \end{thm}

    \begin{proof}
        The case where $\nMat{R}{n}$ has an anti-automorphism follows  from Theorem~\ref{AZ:TH:anti-autos-are-transposes}.
        Let $\alpha$ be an involution of $\nMat{R}{n}$. By Theorem~\ref{AZ:TH:correspondence} and
        Lemma~\ref{AZ:LM:double-prog-factory}, $K_\alpha$ is a double $R$-progenerator admitting an involution.
        Thus, by the proof of Theorem~\ref{AZ:TH:being-mor-eq-to-ring-with-inv}, applied
        with $P=R_R$, $\End_R(R\oplus R^{[1]})$ has
        an involution of type $\sigma$. As in the proof of
        Theorem~\ref{AZ:TH:anti-autos-are-transposes},
        $R^{[1]}\cong (K_\alpha)_1\cong R_R$, so $\End_R(R\oplus R^{[1]})\cong\End_R(R^2)\cong\nMat{R}{2}$ (as
        $C$-algebras).
    \end{proof}

    We now show that if $\nMat{R}{n}$ has an involution, then so does $R$, under further assumptions on $R$.
    To make the statement more general, we introduce \emph{reduced types}.

    Set $\quo{R}:=R/\Jac(R)$ and $\quo{C}:=\Cent(\quo{R})$ ($\quo{C}$ may be larger than the image of $C=\Cent(R)$
    in $\quo{R}$). Let $S$ be a ring such that $\quo{S}:=S/\Jac(S)$ is equipped with
    a central $\quo{C}$-algebra structure (e.g.\ $R^\op$ or $\nMat{R}{n}$).
    If $P$ is an $(S,R)$-progenerator, then by Proposition~\ref{AZ:PR:one-dual-commutes-with-scalar-ext-I},
    $P\Jac(R)=\Jac(S)P$ and $\quo{P}:=P/P\Jac(R)$ is an $(\quo{S},\quo{R})$-progenerator. Thus, $\quo{P}$ induces
    an isomorphism $\Cent(\quo{R})\to\Cent(\quo{S})$. We may view this  isomorphism as an automorphism
    $\quo{\sigma}:\quo{C}\to \quo{C}$ which we call the \emph{reduced type} of $P$.
    Specializing to the case $S=R^\op$, we can define reduced types of $(R^\op,R)$-progenerators and
    double $R$-progenerators.
    In addition, if $\alpha$ is an anti-automorphism of $S$, then the \emph{reduced type} of $\alpha$ is the automorphism
    it induces on $\Cent(\quo{S})\cong\quo{C}$.
    It is easy to check
    that all previous results mentioning types are also valid for reduced types, provided the following convention:
    If $M$ is a right $R$-progenerator and $S=\End_R(M)$, then $\quo{C}$-algebra structure of $\quo{S}$
    is given by the isomorphism $\quo{C}\to \Cent(\quo{S})$  induced
    by the $(\quo{S},\quo{R})$-progenerator $\quo{M}:=M/M\Jac(R)$.

    \begin{thm}\label{AZ:TH:involution-transfer}
        Let $R$ be a semilocal ring such that $\nMat{R}{n}$
        has an involution of type $\sigma\in\Aut(C)$ and reduced type $\quo{\sigma}\in\Aut(\quo{C})$.
        Write
        $\quo{R}:=R/\Jac(R)\cong\prod_{i=1}^t\nMat{D_i}{n_i}$ where $D_i$ is a division ring
        and $n_i\in\N$, and identify
        $\quo{C}=\Cent(\quo{R})$ with $\prod_{i=1}F_i$ where $F_i=\Cent(D_i)$.
        We view each $F_i$ as a non-unital subring of $\quo{C}$.
        Assume that there is $1\leq \ell\leq t$ such that for all $ i\in\{1,\dots,t\}\setminus\{\ell\}$ we have
        \begin{enumerate}
            \item[(1)] $D_i$ is not a field, or $n_i$ is even, or $\quo{\sigma}|_{F_i}\neq\id_{F_i}$,
        \end{enumerate}
        and for the index $\ell$  we have
        \begin{enumerate}
            \item[(2)]
            $D_{\ell}$ is not a field, or $n_{\ell}$ is even, or $\quo{\sigma}|_{F_{\ell}}\neq\id_{F_\ell}$, or $2\in \units{D_{\ell}}$.
        \end{enumerate}
        Then  $R$ has an involution of type $\sigma$ and reduced type $\quo{\sigma}$.
    \end{thm}

    \begin{proof}
        \Step{1} Note first that for an $R$-module
        $M$, $M^n\cong R^n$ implies $M\cong R_R$. Indeed,
        this is clear when $R$ is semisimple, and we can reduce to this case by replacing $M$ with
        $\quo{M}:=M/M\Jac(R)$, as explained in the proof
        of Theorem~\ref{AZ:TH:R-mor-eq-to-R-op}.

        Let $\alpha$ be an involution of $\nMat{R}{n}$ as above. By Theorem~\ref{AZ:TH:anti-autos-are-transposes},
        we may assume $K_\alpha$ is the standard double $R$-module of $(R,\gamma)$ for some anti-automorphism
        $\gamma:R\to R$, necessarily of type (resp.\ reduced type) $\sigma$ (resp.\ $\quo{\sigma}$).
        Furthermore, by Theorem~\ref{AZ:TH:correspondence}, $K_\alpha$ has an involution $\theta$.
        Namely, $\theta:R\to R$ is a map satisfying
        \begin{equation}\label{AZ:EQ:theta-relation}
        (a^\gamma bc)^\theta=c^\gamma b^\theta a\qquad \forall a,b,c\in R.
        \end{equation}

        Suppose that $\theta$ fixes an invertible element $u\in \units{R}$. We claim that the map $\beta:R\to R$
        defined by $r^\beta=u^{-1}r^\gamma u$ is an involution, clearly of the same type and reduced type as $\gamma$.
        Indeed, observe that for all $r\in R$,
        $(ur)^\theta=r^\gamma u^\theta=uu^{-1}r^\gamma u=ur^\beta$. Therefore,
        $ur =(ur)^{\theta\theta}=(ur^\beta)^\theta=ur^{\beta\beta}$, which implies $r=r^{\beta\beta}$.

        We are thus reduced to show that there exists $u\in\units{R}$ with $u^\theta=u$. Likewise, we are also done if
        we can find $u\in\units{R}$ with $u^\theta=-u$; simply replace $\theta$ with $-\theta$.

        \Step{2}
        Let $J=\Jac(R)$. Clearly $J^\gamma= J$, hence $\gamma$ induces an anti-automorphism on $\quo{R}$.
        We claim that $J^\theta\subseteq J$. Indeed, $J^\theta$ is an ideal
        by \eqref{AZ:EQ:theta-relation} and the bijectivity of $\gamma$, so  is enough to show that $1+J^\theta=(1^\theta+J)^\theta$
        consists of left invertible elements.
        Indeed, $1=(1\cdot 1)^{\theta\theta}=(1\cdot 1^\theta)^\theta=(1^\theta)^\gamma\cdot 1^\theta$, hence $1^\theta$ is left invertible, and
        a simple variation shows that $1^\theta$ is also right invertible.
        Thus, $1^\theta+J$ consists of invertible elements.
        Now, $\theta$ sends right invertible elements to left invertible elements since  $sr=1$ implies
        $r^\gamma s^\theta =(sr)^\theta=1^\theta\in\units{R}$. Therefore, $(1^\theta+J)^\theta$ consists of left invertible
        elements, as required.

        Let $\quo{\gamma},\quo{\theta}$ be the maps induced by $\gamma,\theta$ on $\quo{R}=R/J$.
        In the following steps, we shall establish the existence of $x\in R$ with $x+x^\theta\in\units{R}$ or $x-x^\theta\in\units{R}$.
        This would finish the proof since we can then take $u$ of Step~1 to be $x\pm x^\theta$. It is enough to prove
        the existence of $x$ for $\quo{R}$, $\quo{\gamma}$, $\quo{\theta}$ and then lift it arbitrarily to $R$. Thus, we may henceforth
        assume $R$ is semisimple.

        \Step{3}
        Let $C=\Cent(R)$. Then $C^\gamma=C$. Moreover, for all $c\in C$,
        \[c=c^{\theta\theta}=((1c)^\theta)^\theta=(c^\gamma 1^\theta)^\theta=(1^\theta c^\gamma)^\theta=
        c^{\gamma\gamma}1^{\theta\theta}=c^{\gamma\gamma}\ ,\]
        so $\gamma|_C$ is an involution. Let $\{e_1,\dots,e_t\}$ be the primitive idempotents of $C$.
        As $R$ is semisimple, we may assume that $e_iRe_i=e_iR=\nMat{D_i}{n_i}$ (with $D_i$, $n_i$ as in the theorem).
        Now, $\gamma$ permutes $\{e_1,\dots,e_t\}$ and the orbits have size $1$ or $2$.
        Choose representatives $\{e_i\}_{i\in I}$ for the orbits of $\gamma$ such that $1\in I$,
        and let $f_i=e_i+e_i^\gamma$ if $e_i\neq e_i^\gamma$ and $f_i=e_i$ otherwise.
        Then $\{f_i\}_{i\in I}$ are pairwise orthogonal, $\sum_{i\in I} f_i=1$ and $f_i^\gamma=f_i$ for all $i$.

        Observe that if $f\in R$ satisfies $f^\gamma=f$, then $(fRf)^\theta=f^{\gamma}R^\theta f^{\gamma^{-1}}=fRf$. In particular,
        $\theta$ takes $f_iRf_i=f_iR$ into itself for all $i$.
        As $R=\prod f_iR$, it is enough to find
        $x_i,x'_i\in f_iR$ such that $x_i+x_i^\theta,x'_i-x'^\theta_i\in \units{(f_iR)}$
        for all $i\neq\ell$, and at least one of $x_\ell+x_\ell^\theta$, $x'_\ell-x'^\theta_\ell$ is a unit of $f_\ell R$.
        (Indeed, take $x=\sum_{i\in I} x_i$ if $x_\ell+x_\ell^\theta\in\units{(f_1R)}$
        and take $x=\sum_{i\in I} x'_i$ if $x'_\ell-x'^\theta_\ell\in\units{(f_1R)}$.)
        We now split into cases.

        \Step{4} Suppose $f_i=e_i+e_i^\gamma$ with $e_i\neq e_i^\gamma$.
        Then $f_iR= e_iR\times e_i^\gamma R$.
        Observe that $e_i^\theta=(e_ie_i)^\theta=e_i^{\gamma}e_i^\theta\in e_i^\gamma R$.
        We claim that $e_i^\theta$ is invertible in $e_i^\gamma R$.
        Indeed, we have
        \begin{align*}
        e_i^\gamma&=e_i^{\gamma\theta\theta}=(e_i^\gamma e_i^\gamma)^{\theta\theta}=
        (e_i^{\gamma\gamma}e_i^{\gamma\theta})^\theta=e_i^{\gamma\theta\gamma}e_i^{\gamma\gamma\theta}=
        e_i^{\gamma\theta\gamma}e_i^\theta\ ,\\
        e_i^\gamma&=e_i^{\gamma^{-1}\theta\theta}=(e_i^{\gamma^{-1}}e_i^{\gamma^{-1}})^{\theta\theta}=
        (e_i^{\gamma^{-1}\theta} e_i^{\gamma^{-1}\gamma^{-1}})^\theta=
        e_i^{\gamma^{-1}\gamma^{-1}\theta}e_i^{\gamma^{-1}\theta\gamma^{-1}}=e_i^\theta e_i^{\gamma^{-1}\theta\gamma^{-1}}\ ,
        \end{align*}
        hence $e_i^\theta$ is left and right invertible in $e_i^\gamma R$ (we used \eqref{AZ:EQ:theta-relation}
        and the fact that $e_i^{\gamma\gamma}=e_i$ repeatedly).
        Since $e_i$ is invertible in $e_iR$, it follows that $e_i+e_i^\theta$ and $e_i-e_i^\theta$ are invertible
        in $f_iR= e_iR\times e_i^\gamma R$. We may therefore take $x_i=x'_i=e_i$.

        \Step{5} Assume  that $f_i=e_i$, and let $S:=f_iR=e_iR=\nMat{D_i}{n_i}$.
        By Theorem~\ref{AZ:TH:anti-autos-are-transposes}, there is an anti-automorphism $\beta:D_i\to D_i$
        and $v\in \units{S}$ such that $v^{-1}s^\gamma v=s^{T_{n_i}\beta}$ for all $s\in S$.
        Let $\gamma'=T_{n_i}\beta$ and define $\theta':S\to S$ by $s^{\theta'}=v^{-1}s^\theta v^{\gamma^{-1}}$.
        It is straightforward
        to check that $(a^{\gamma'}bc)^{\theta'}=c^{\gamma'}b^{\theta'}a$ and $a^{\theta'\theta'}=a$
        for all $a,b,c\in S$. Furthermore, for all $x\in S$, $x\pm x^{\theta'}=v^{-1}(vx+(vx)^\theta)$.
        Therefore, without loss of generality, we may replace $\gamma$, $\theta$ with $\gamma'$, $\theta'$.

        Let $\{e_{jk}\}$ be the standard $D_i$-basis of $S=\nMat{D_i}{n_i}$. Then  $\gamma=T_{n_i}\beta$ implies
        $e_{jk}^{\gamma}=e_{kj}$. Therefore, $(e_{jj}Se_{kk})^\theta=e_{kk}^{\gamma}S^\theta e_{jj}^{\gamma^{-1}}=e_{kk}Se_{jj}$, namely,
        $\theta$ takes $e_{jk}D_i$ to $e_{kj}D_i$.
        Now, if $n_i$ is even,  we can take $x_i=x'_i=\sum_{j=1}^{n_i/2}e_{(2j)(2j-1)}$. Indeed, both $x_i+x_i^\theta$
        and $x'_i-x'^\theta_i$ have matrices of the form
        \[
        \left[\begin{smallmatrix}
        0 & * \\
        * & 0 \\
          &   & 0 & * \\
          &   & * & 0 \\
          &   &   &   & \ddots
        \end{smallmatrix}\right]
        \]
        with all the $*$-s being nonzero, so they are invertible.

        Next, embed $D:=D_i$ in $S$ by sending $d\in D$ to the diagonal matrix with diagonal entries $(d,\dots,d)$. Since $\gamma=T_{n_i}\beta$,
        the anti-automorphism $\beta$ of $D$ coincides with $\gamma|_D$. We  claim that $\theta$
        sends $D$ into itself. Indeed, as a subring of $S$,
        $D$ is the centralizer of $\{e_{jk}\}$, and for all $d\in D$ and $1\leq j,k\leq n_i$, we have
        \[
        e_{jk}d^\theta=e_{kj}^\gamma d^\theta=(de_{kj})^\theta=(e_{kj}d)^\theta=d^\theta e_{kj}^{\gamma^{-1}}=d^\theta e_{jk}\ .
        \] 
        It is therefore enough to find $x_i,x'_i\in D$ with $x_i+x_i^\theta,x'_i-x'^\theta_i\in \units{D}$.
        Since $D$ is a division ring,
        the non-existence of $x_i$ or $x'_i$ forces $\theta=\pm\id_D$, which in turn implies $\beta=\gamma|_D=\id_D$,
        because then
        $\pm d=d^\theta=(1d)^\theta =d^\gamma1^\theta=\pm d^\gamma$. This is impossible in case $D_i$ is not a field
        or $\quo{\sigma}$ (which is the restriction of $\gamma$ to $\Cent(R)$) does not fix $F_i=\Cent(S)=\Cent(D)$ pointwise.

        In case $i\neq \ell$, either $n_i$ is even, or $D_i$ is not a field, or $\quo{\sigma}$ does not fix $F_i$ element-wise,
        so the existence of both $x_i$ and $x'_i$ is guaranteed by
        the previous paragraphs. In case $i=\ell$, $n_i$ is odd,  $D_i$ is a field, and $\quo{\sigma}$ fixes $F_j$ element-wise,
        we must have $\Char D_i\neq 2$. Returning to the
        setting of the previous
        paragraph, this means that for any $0\neq y\in D$, at least one of $y+y^\theta$, $y-y^\theta$ is nonzero (for their sum
        is $2y\neq 0$). Thus, either $x_1$ or $x'_1$ exists, as required.
    \end{proof}

    \begin{cor}\label{AZ:CR:basic-ring-involution}
        Let $R$ be a semiperfect ring having an involution of type $\sigma$ and reduced type $\quo{\sigma}$.
        Write $\quo{R}:=R/\Jac(R)\cong\prod_{i=1}^t\nMat{D_i}{n_i}$ as in Theorem~\ref{AZ:TH:involution-transfer}
        and assume that there is $1\leq \ell\leq t$ such that for all $ i\in\{1,\dots,t\}\setminus\{\ell\}$ we have
        \begin{enumerate}
            \item[(1$'$)] $D_i$ is not a field, or $\quo{\sigma}|_{F_i}\neq\id_{F_i}$,
        \end{enumerate}
        and for the index $\ell$  we have
        \begin{enumerate}
            \item[(2$'$)]
            $D_{\ell}$ is not a field, or $\quo{\sigma}|_{F_{\ell}}\neq\id_{F_\ell}$, or $2\in \units{D_{\ell}}$.
        \end{enumerate}
        Then the basic ring of $R$ has an involution of  type $\sigma$ and reduced type $\quo{\sigma}$.
    \end{cor}

    \begin{proof}
        Let $S$ be the basic ring of $R$. By Proposition~\ref{AZ:PR:involutions-pass-to-the-basic-ring}, $\nMat{S}{2}$
        has an involution of type $\sigma$
        and reduced type $\quo{\sigma}$. Using the description $S$ given in section~\ref{section:semiperfect}
        and Proposition~\ref{AZ:PR:one-dual-commutes-with-scalar-ext-I}, it is easy to see that $\quo{S}$
        is the basic ring of $\quo{R}$. That is, $\quo{S}\cong \prod_{i=1}^tD_i$. Therefore, $S$ satisfies
        the conditions of
        Theorem~\ref{AZ:TH:involution-transfer}, and hence $S$ has an involution of type $\sigma$ and reduced type $\quo{\sigma}$.
    \end{proof}

    \begin{cor}
        Let $R$ be a finite-dimensional algebra over a field $F$. Suppose $\nMat{R}{n}$ admits an involution
        $\alpha$ sending $F=F\cdot 1_R$ into itself and satisfying $\alpha|_F\neq \id_F$. Then $R$ has an involution of
        the same type and reduced type as $\alpha$. The same applies to the basic $F$-algebra that is Morita equivalent
        to $R$.
    \end{cor}

    \begin{proof}
        Let $a\in F$ be an element with $a^\alpha- a\neq 0$.
        Write $\quo{C}=\Cent(R/\Jac(R))=F_1\times\dots \times F_n$ with each $F_i$
        a field, and let $\quo{\sigma}\in\Aut(\quo{C})$ be the reduced type of $\alpha$.
        Pick some $0\neq x\in F_i$ and view it as an element of $\quo{C}$. If both $\quo{\sigma}(x)=x$,
        and $\quo{\sigma}(ax)=ax$, then
        $ax={\quo{\sigma}}(ax)=\quo{\sigma}(a) \quo{\sigma}(x)=a^\alpha x$, which implies $(a-a^\alpha)x=0$,
        a contradiction. Thus, $\quo{\sigma}$ does not fix any of the fields $F_i$ element-wise, so we are done
        by Theorems~\ref{AZ:TH:involution-transfer} and Corollary~\ref{AZ:CR:basic-ring-involution}.
    \end{proof}

    \begin{remark}
        In general, if $R$, $\theta$, $\gamma$ are as in the proof of Theorem~\ref{AZ:TH:involution-transfer},
        then there may not exist  $u\in \units{R}$ with $u=u^\theta$ or $u=-u^\theta$. For example,
        take $R=F\times F$ with $F$ a field of characteristic not $2$, let $\alpha=\id_R$, and let $\theta$ be defined by
        $(x,y)^\theta=(x,-y)$. This implies that the proof of  Theorem~\ref{AZ:TH:involution-transfer}
        cannot be applied to arbitrary semilocal rings. In fact, we believe that
        there exist semilocal rings $R$ without involution such that $\nMat{R}{2}$ has an involution.
    \end{remark}

    \begin{remark}
        Given $R$, the datum of $\theta$ and $\gamma$ satisfying \eqref{AZ:EQ:theta-relation}
        is equivalent to the datum of an \emph{anti-structure}, a notion introduced by
        Wall  \cite[pp.\ 244]{Wall70} to define quadratic
        forms over rings. Recall that an anti-structure on $R$ consists
        of an anti-automorphism $\gamma:R\to R$ and an element $v\in \units{R}$ such that $v^\gamma=v^{-1}$
        and $r^{\gamma\gamma} =vrv^{-1}$ for all $r\in R$.
        Fixing $\gamma$, there is a one-to-one correspondence between the possible $\theta$-s
        and $v$-s given by $v=\theta(1)$ and $r^\theta=r^\gamma v=vr^{\gamma^{-1}}$.
        Therefore, the proof of Theorem~\ref{AZ:TH:involution-transfer-weak} shows that if $R$ has an anti-structure, then $\nMat{R}{2}$
        has an involution. Unfolding the construction, this involution is given by
        \[
        \SMatII{a}{b}{c}{d}\mapsto \SMatII{d^\gamma}{b^\theta }{v^{-1} c^\theta v^{-1}}{a^{\gamma^{-1}}}=
        \SMatII{d^\gamma}{ b^\gamma v}{v^{-1}c^\gamma}{a^{\gamma^{-1}}}
        \]
        This trick was noted by several authors in the past (see  \cite[pp.\ 532]{Sa78}, for instance).
        The proof of theorem~\ref{AZ:TH:being-mor-eq-to-ring-with-inv} can be viewed as a generalization of it.

        This means that the question of whether there is a semilocal ring $R$ without involution such that $\nMat{R}{2}$
        has an involution is equivalent to whether there is a semilocal ring without an involution admitting an anti-structure.
        In fact, we were unable to find in the literature an example of a \emph{general} ring with these properties.
    \end{remark}

\section{Counterexamples}
\label{section:examples}

    We now demonstrate that there are  ``nice'' rings which are not Morita equivalent to rings with
    an involution (of any kind), but still admit an anti-auto\-mor\-phism. In the first example, that auto\-mor\-phism
    fixes the center element-wise.

    \begin{example} \label{FORM:EX:two-nderives-one}
        Recall that a \emph{poset} consists of a finite set $I$ equipped with a partial order
        which we denote by $\leq$.
        For a field $F$ and a poset $I$, the \emph{incidence algebra} $F(I)$ is defined to be
        the subalgebra of the $I$-indexed matrices over $F$
        spanned as an $F$-vector space by the matrix units $\{e_{ij}\where i,j\in I,\ i\leq j\}$.
        If $I$ is not the disjoint union of two non-comparable subsets, then $\Cent(F(I))=F$. In
        this case, we say $I$ is \emph{connected}. The algebra $F(I)$ is well-known to be basic.

        Let $R$ be a ring that is Morita equivalent to $F(I)$. The poset $I$ can be recovered
        from $R$, up to isomorphism, as follows:
        Let $E$ denote the set of primitive idempotents in $R$.
        Then $\units{R}$ acts by conjugation on $E$. Define ${I}_R$ to be the set of equivalence classes in $E$, and
        for $i,j\in {I}_R$, let $i\leq j$ if and only if $eRf\neq 0$ for some (and hence any) $e\in i$ and $f\in j$.
        To see that ${I}_R$ is indeed isomorphic to $I$, observe that there is an isomorphism
        between $I$ and the isomorphism classes of indecomposable projective $R$-modules
        via $i\mapsto P_i:=eR$ ($e\in i$), and $i\leq j$ if and only if $\Hom_R(P_j,P_i)\neq 0$.
        This shows that ${I}_R$ can be determined from $\rMod{R}$, hence $I_R\cong I_{F(I)}$. Finally, $I_{F(I)}$ is
        easily seen to be isomorphic to $I$ (via sending the class of $e_{ii}\in F(I)$ to $i$).
        The construction of $I_R$ implies that any anti-automorphism (resp.\ involution) of $R$ induces an anti-automorphism
        (resp.\ involution)
        on ${I}\cong I_R$.

        Suppose now that $I$ is connected, has an anti-automorphism, but admits no involution. Then the previous paragraphs imply that
        $\Cent(F(I))=F$, and any ring that is Morita equivalent to $F(I)$ does not have an involution. On the other
        hand,
        the anti-automorphism of $I$ gives rise to an anti-automorphism of $F(I)$ of type $\id_F$.
        An example of such a poset  was given in \cite{Sch74} by Scharlau (for other purposes); $I$ is
        the $12$-element poset whose Hasse
        diagram is:
        \[
        \xymatrix{
                              & \bullet \ar[r] & \bullet        & \\
        \bullet \ar[ur]\ar[d] & \bullet \ar[u] & \bullet        & \bullet \ar[ul]\ar[l] \\
        \bullet \ar[dr]\ar[r] & \bullet        & \bullet \ar[d] & \bullet \ar[u]\ar[dl] \\
                              & \bullet        & \bullet \ar[l]
        }
        \]
        (Using Scharlau's words, it is ``the simplest example I could find''.) An anti-automorphism
        of $I$ is given by rotating the diagram  ninety degrees clockwise.

        Involutions and automorphisms of incidence algebras are well-understood in general; we refer the reader
        to the survey \cite{BruLew11} and the references therein.
    \end{example}

     \begin{example}\label{AZ:EX:div-ring-without-involution}
        Various f.d.\ division algebras admitting an anti-automorphism
        but no involution were constructed in \cite{MoranSethTig05}.
        By Proposition~\ref{AZ:PR:inv-over-local-rings}, none of these algebras is Morita equivalent to a ring with involution.
        Note that by Albert's Theorem (\cite[Th.\ 10.19]{Al61StructureOfAlgs}),  anti-automorphisms
        of such division algebras cannot fix the center pointwise.
    \end{example}

\section{Azumaya Algebras}
\label{section:azumaya}

    The rest of this paper concerns Saltman's Theorem about Azumaya algebras
    with involution (see section~\ref{section:overview}) and the construction of Azumaya algebras without involution.
    As preparation, we now briefly recall Azumaya algebras and several facts about them to be used later.
    We refer
    the reader to \cite{Sa99} and \cite{DeMeyIngr71SeparableAlgebras} for
    an extensive discussion and proofs. Throughout, $C$ is a commutative ring and, unless specified otherwise, all tensor
    products are taken over $C$.

\subsection{Azumaya Algebras}

    A  $C$-algebra $A$ is called \emph{Azumaya} if $A$ is a progenerator as a $C$-module and the standard
    map $\Psi:A\otimes A^\op\to \End_C(A)$ given by $\Psi(a\otimes b^\op)(x)=axb$ is an isomorphism. When $C$ is a
    field, being Azumaya is equivalent to being simple and central, so Azumaya algebras are a generalization of central simple algebras.

    \smallskip

    Let $A$, $B$ be Azumaya $C$-algebras. The following facts are well-known:
    \begin{enumerate}
        \item $\Cent(A)=C$ and $\Cent_{A\otimes B}(C\otimes B)=A\otimes C$.
        \item $A\otimes B$ and $A^\op$ are Azumaya $C$-algebras.
        \item If $\psi:C\to C'$ is a commutative ring homomorphism,
        then $A\otimes C'$ is an Azumaya $C'$-algebra ($C'$ is viewed as a $C$-algebra via $\psi$).
        \item For every $C$-progenerator $P$, $\End_C(P)$ is an Azumaya $C$-algebra.
        \item If $B$ is a subalgebra of $A$, then
        $B':=\Cent_A(B)$ is Azumaya (over $C$), $B=\Cent_A(B')$ and $B\otimes B'\cong A$ via $b\otimes b'\mapsto bb'$.
    \end{enumerate}

    For every $M\in\rproj{C}$, we define $\rank(M)=\rank_C(M)$ to be the function $\Spec(C)\to \Z$ sending
    a prime ideal $P$ to $\dim_{k_P}(M\otimes k_P)$ where $k_P$ is the fraction field of $C/P$.
    We write $\rank(M)=n$ to denote that $\rank(M)(P)=n$ for all $P\in\Spec(C)$. For example, when $C$
    is a field, $\rank(M)=\dim_C(M)$.
    More generally, $\rank(M)$ is constant when $\Spec(C)$ is connected (as a topological space),
    or equivalently, when $C$ does not decompose as product of two nonzero rings.
    In fact, we can always write $C=C_1^{(M)}\times\dots\times C_t^{(M)}$ such that
    $\rank(M\otimes C_i^{(M)})$ is constant. (This follows from the well-known facts
    that $\rank(M):\Spec(C)\to \Z$ is continuous,
    and $\Spec(C)$ is a compact topological space.)
    Lastly, observe that if $\psi:C\to C'$ is a commutative ring homomorphism and $P'\in\Spec(C')$,
    then $\rank_{C'}(M\otimes C')(P')=\rank_{C}(M)(P)$ where $P=\psi^{-1}(P')$. This allows us to extend scalars
    when computing ranks.

    \begin{prp}\label{AZ:PR:hom-between-azumaya-algs}
        Let $A$ and $B$ be Azumaya $C$-algebras and let $\vphi:A\to B$ be a homorphism of $C$-algebras.
        Then $\vphi$ is injective. If moreover $\rank(A)=\rank(B)$, then $\vphi$ is an isomorphism.
    \end{prp}

    Two Azumaya $C$-algebras $A,B$ are said to be \emph{Brauer equivalent}, denoted \mbox{$A\Breq B$}, if
    there are $C$-progenerators $P,Q$ such that
    \[
    A\otimes \End_C(P)\cong B\otimes \End_C(Q)
    \]
    as $C$-algebras.
    The \emph{Brauer class} of $A$, denoted $[A]$, is the collection of Azumaya algebras which are Brauer equivalent
    to $A$. The \emph{Brauer group} of $C$, denoted $\Br(C)$, is the set of
    all Brauer equivalence classes endowed with the group operation $[A]\otimes [B]=[A\otimes B]$.
    The unit element of $\Br(C)$ is the class of $C$, namely, the class of Azumaya algebras which are isomorphic to
    $\End_C(P)$ for some progenerator $P$. The inverse of $[A]$ is $[A^\op]$.

    The following theorem, which is essentially due to Bass, presents an alternative definition of the Brauer equivalence.

    \begin{thm}[Bass]\label{AZ:TH:Bass-criterion}
        Let $A,B$ be two Azumaya $C$-algebras. Then $A\Breq B$ $\iff$ $A\CMoreq{C} B$, i.e.\ there exists an $(A,B)$-progenerator
        $P$ such that $cp=pc$ for all $p\in P$ and $c\in C$.
    \end{thm}

    \begin{proof}
        See \cite[Cor.\ 17.2]{Ba64}. The assumption that $\Spec(C)$ is noetherian in \cite{Ba64} can be ignored
        by \cite[Th.\ 2.3]{Sa99}; see also \cite[\S1]{Sa78}. 
    \end{proof}

    The \emph{exponent} of an Azumaya algebra $A$ is its order in $\Br(C)$.

    \begin{thm}\label{AZ:TH:exponent-rank-relation}
        Let $A$ be an Azumaya $C$-algebra. Then $\rank(A)$ is a square. Furthermore, if $\rank(A)$ divides $n^2$,
        then $A^{\otimes n}\Breq C$.
    \end{thm}

    We will also need the following proposition.

    \begin{prp}\label{AZ:PR:progenerators-between-Azu-algs}
        Let $A$, $B$ be Azumaya $C$-algebras and let $K$ be an $(A,B)$-bimodule
        satisfying $ck=kc$ for all $c\in C$. Then $K$ is an $(A,B)$-progenerator
        if and only if $K$ is a $(C,A^\op\otimes B)$-progenerator
        with $A^\op\otimes B$ acting via $k(a^\op\otimes b)=akb$.
        In particular, in this case, $\End_C({}_CK)\cong A^\op\otimes B$.
    \end{prp}

    \begin{proof}
        The ``only if'' part easily follows from  $\Cent_{A^\op\otimes B}(C\otimes B)=A^\op\otimes C$,
        so we turn to prove the ``if'' part.
        Let $D=A\otimes A^\op$ and consider $A$ as a right $D$-module via $a(x\otimes y^\op)=yax$.
        Since ${}_CA$ is a progenerator and $D\cong \End({}_CA)$ (because $A$ is Azumaya),
        $A$ is a $(C,D)$-progenerator.
        By Proposition~\ref{AZ:PR:scalar-ext-of-prog}, $K\otimes A^\op$ is an $(A\otimes A^\op,B\otimes A^\op)$-progenerator
        and hence ${}_CA\otimes_D(K\otimes A^\op)$ is a $(C,B\otimes A^\op)$-progenerator.
        It is straightforward to check that ${}_CA\otimes_D(K\otimes A^\op)\cong K$ as $(C,B\otimes A^\op)$-bimodules
        via $x\otimes_D(k\otimes y^\op)\mapsto yxk$ and $k\mapsto 1\otimes_D(k\otimes 1)$ in the other direction,
        so we are done.
    \end{proof}

\subsection{Galois Extensions}

    Let $G$ be a finite group of ring-automorphisms of $C$ and let
    $C_0=C^G$ be the ring of elements fixed by $G$. For every $2$-cocycle $f\in Z^2(G,\units{C})$,\footnote{
        Recall that a $2$-cocycle of $G$ taking values in $\units{C}$ is a function $f:G\times G\to \units{C}$
        satisfying $\sigma(f(\tau,\eta))f(\sigma,\tau\eta)=f(\sigma,\tau)f(\sigma\tau,\eta)$ and
        $f(1,\sigma)=f(\sigma,1)=1$ for all $\sigma,\tau,\eta\in G$.
        The $2$-cocycles with the operation of point-wise multiplication form an abelian group denoted $Z^2(G,\units{C})$.
    }
    let $\Delta(C/C_0,G,f)=\bigoplus_{\sigma\in G}Cu_\sigma$ where $\{u_\sigma\}_{\sigma\in G}$
    are formal variables. We make $\Delta(C/C_0,G,f)$ into a ring by linearly extending
    the following relations
    \[
            u_\sigma u_\tau=f(\sigma,\tau)u_{\sigma\tau},\qquad u_\sigma c=\sigma(c)u_\sigma,\qquad\forall \sigma,\tau\in G,~c\in C\ .
    \]
    This makes $\Delta(C/C_0,G,f)$ into a $C_0$-algebra whose unity is $u_1$.

    We say that $C/C_0$ is a \emph{Galois extension with Galois group $G$} if
    $\Delta(C/C_0,G,1)$ is an Azumaya $C_0$-algebra ($1$ denotes the trivial $2$-cocycle $f(\sigma,\tau)\equiv 1$).
    In this case:
    \begin{enumerate}
        \item $\Delta(C/C_0,G,1)\cong \End_{C_0}(C)$ via  $\sum_\sigma c_\sigma u_\sigma \mapsto
        [x\mapsto \sum_\sigma c_\sigma \sigma(x)]\in\End_{C_0}(C)$.
        \item The $C_0$-algebra $\Delta(C/C_0,G,f)$ is Azumaya for all $f\in Z^2(G,\units{C})$.
        \item $\rank_{C_0}(C)=|G|$.
        \item For every commutative $C_0$-algebra $D$, the extension $C\otimes_{C_0} D/C_0\otimes_{C_0} D$ is
        Galois  with Galois group  $\{g\otimes_{C_0} \id_D\where g\in G\}$.
    \end{enumerate}
    See \cite[\S6]{Sa99} for proofs.
    Azumaya algebras of the form $\Delta(C/C_0,G,f)$ are called \emph{crossed products}.

    \begin{prp}\label{AZ:PR:Galois-invariant-part}
        Let $C/C_0$ be a Galois extension with Galois group $G$. Assume $M$ is a $C$-module
        endowed with a $G$-action such that $\sigma(mc)=\sigma(m)\sigma(c)$ for all $\sigma\in G$, $m\in M$, $c\in C$.
        Then $M^G\otimes_{C_0}C\cong M$ via $m\otimes c\mapsto mc$, where
        \[M^G:=\{m\in M\suchthat \sigma(m)=m~\forall \sigma\in G\}\ .\]
        In particular, $\rank_{C_0}(M)=\rank_{C_0}(C)\rank_{C_0}(M^G)=|G|\cdot\rank_{C_0}(M^G)$.
        Furthermore, if $M$ is an Azumaya $C$-algebra and $G$ acts via ring automorphisms on $M$, then
        $M^G$ is an Azumaya $C_0$-algebra.
    \end{prp}

    \begin{proof}
        See \cite[Prps.\ 6.10 \& 6.11]{Sa99}.
    \end{proof}

\subsection{Corestriction}

    Let $C/C_0$ be a Galois extension with Galois group $G$. For every Azumaya $C$-algebra $A$ and $\sigma\in G$,
    define $A^\sigma$ to be the $C$-algebra obtained from $A$ by viewing $A$ as a $C$-algebra
    via $\sigma:C\to A$. The algebra $A^\sigma$ is also Azumaya, but it need not be Brauer-equivalent to $A$.
    Observe that the algebra
    $B:=\bigotimes_{\sigma\in G}A^\sigma$ admits a $G$-action given by $\tau(\bigotimes_\sigma a_\sigma)=\bigotimes_\sigma a_{\tau^{-1}\sigma}$
    and that action satisfies $\sigma(bc)=\sigma(b)\sigma(c)$ for all $\sigma\in G$, $b\in B$, $c\in C$.
    The \emph{corestriction} of $A$ (with respect to $C/C_0$) is defined to be
    \[
        \Cor_{C/C_0}(A)=B^G=\Big(\bigotimes_{\sigma\in G}A^\sigma\Big)^G\ .
    \]
    By Proposition~\ref{AZ:PR:Galois-invariant-part},  $\Cor_{C/C_0}(A)$ is an Azumaya $C_0$-algebra. Moreover, it
    can be shown that the corestriction induces a group homomorphism
    \[
        \Cor_{C/C_0}~:~\Br(C)\to\Br(C_0)
    \]
    given by $\Cor_{C/C_0}([A])=[\Cor_{C/C_0}(A)]$.
    We also note  that the corestriction can be defined for  \emph{separable} extensions $C/C_0$; see \cite[Ch.\ 8]{Sa99} for
    further details.

\section{Saltman's Theorem}
\label{section:first-kind}

    In this section, we show how
    to recover the Knus-Parimala-Srinivas proof of Saltman's Theorem (\cite[\S4]{KnParSri90})
    as an application of Theorem~\ref{AZ:TH:being-mor-eq-to-ring-with-inv}.
    In order to keep the exposition as self-contained and fluent as possible, we will first
    give a proof of Saltman's Theorem using Theorem~\ref{AZ:TH:being-mor-eq-to-ring-with-inv}, and then
    explain how this proof relates to the proof in \cite{KnParSri90}.

\subsection{Saltman's Theorem}

    Let $C$ be a commutative ring and let $R$ be a $C$-algebra.
    Recall that
    a Goldman element of $R/C$ is an element $g\in R\otimes_CR$ such
    that $g^2=1$ and $g(r\otimes s)=(s\otimes r)g$ for all $r,s\in R$.

    \begin{prp}\label{AZ:EX:Azumayma-alg-has-Goldman-elem}
        All Azumaya $C$-algebras have a Goldman element
    \end{prp}

    \begin{proof}
        See \cite[p.\ 112]{KnusOjan74} or \cite[Pr.\ 5.1]{Sa99}, for instance.
    \end{proof}

    \begin{thm}[Saltman]\label{AZ:TH:Saltman}
        Let $C$ be a commutative ring and let $A$ be an Azumaya $C$-algebra. Then:
        \begin{enumerate}
            \item[(i)] $A$ is Brauer equivalent to an Azumaya algebra $B$ with an involution
            of the first kind if and only if $A\Breq A^\op$.
            \item[(ii)] Let $C/C_0$ be a Galois extension with Galois group $\{1,\sigma\}$ ($\sigma\neq 1$).
            Then $A$ is Brauer equivalent to an Azumaya algebra $B$ with an involution whose restriction
            to $C$ is $\sigma$ if and only if $\Cor_{C/C_0}(B)\Breq C_0$.
        \end{enumerate}
    \end{thm}

    \begin{proof}
        (i)
        By Theorem~\ref{AZ:TH:Bass-criterion}, $A\Breq A^\op$ if and only if there exists a double $A$-progenerator
        of type $\id_C$ (i.e.\ an $(A^\op,A)$-progenerator of type $\id_C)$.
        Likewise,
        by Theorems~\ref{AZ:TH:being-mor-eq-to-ring-with-inv} and~\ref{AZ:TH:Bass-criterion},
        $A$ is Brauer equivalent to a $C$-algebra $B$ with
        involution of the first kind if and only if there exist a double $A$-progenerator of type $\id_C$ with involution.
        Therefore, it is enough to show that every double $A$-progenerator $K$ of type $\id_C$ has an involution.
        Indeed, consider $K$ as an $A\otimes_C A$-module via $k\cdot(a\otimes_C a')=k\mul{0}a\mul{1}a'$
        ($a,a'\in A$) and let $g$ be a Goldman element of $A$.
        It is easy to check that  $k\mapsto kg$ is an involution of $K$.

        (ii) Throughout, $D=A\otimes_C A^\sigma$ and $E=D^G=\Cor_{C/C_0}(D)$.
        Recall that $\sigma$ extends to an automorphism of $D$, also denoted  $\sigma$,
        given by $\sigma(a\otimes a')= a'\otimes a$.
        By Theorems~\ref{AZ:TH:being-mor-eq-to-ring-with-inv} and~\ref{AZ:TH:Bass-criterion},
        it is enough to prove that $\Cor_{C/C_0}(A)\Breq C_0$ if and only if then there exists  a double $A$-progenerator
        of type $\sigma$
        admitting an involution.

        Assume $K$ is a double $A$-progenerator of type $\sigma$ with involution $\theta$.
        Then $K$ can be considered as an $(A^\op,A^\sigma)$-progenerator of type $\id_C$, hence by
        Proposition~\ref{AZ:PR:progenerators-between-Azu-algs}, $\End({}_CK)\cong A\otimes A^\sigma=D$.
        Extend the left action of $C$ on $K$ to a left action of $\Delta:=\Delta(C/C_0,G,1)$ on $K$ by
        letting $u_\sigma$ act as $\theta$. Then it is easy to see that $\End({}_{\Delta}K)\cong D^G=E$.
        Viewing $\Delta^\op$ and $E$ as subrings of $\End({}_{C_0}K)$, this means $E$ is the centralizer
        of $\Delta^\op$ in $\End({}_{C_0}K)$. As all these algebras are Azumaya over $C_0$,
        we have $\Delta^\op\otimes_{C_0} E\cong \End({}_{C_0}K)$. In $\Br(C_0)$ this reads as $[C_0]\otimes [E]=[\Delta^\op]\otimes [E]=[C_0]$,
        so $\Cor_{C/C_0}(A)=E\Breq C_0$.

        Conversely, assume $E=\Cor_{C/C_0}(A)\Breq C_0$.
        Then there exists a $(C_0,E)$-progenerator $Q$ of type $\id_{C_0}$. Let $K=Q\otimes_{C_0}C$.
        Then by Proposition~\ref{AZ:PR:standard-hom-III}, $K$ is a $(C_0\otimes_{C_0} C,E\otimes_{C_0} C)$-progenerator.
        Identify $E\otimes_{C_0}C$ with $D$ via $(e\otimes c)\mapsto ec$. Then
        $K$ is a $(C,D)$-progenerator, hence by Proposition~\ref{AZ:PR:progenerators-between-Azu-algs},
        $K$ is an $(A^\op, A^\sigma)$-progenerator
        via $a^\op\cdot k\cdot a':=k(a\otimes a')$. We may thus view $K$
        as a double $A$-progenerator of type $\sigma$.
        Let $\theta=\sigma\otimes_{C_0}\id_Q:K\to K$ (so $(c\otimes q)^\theta=\sigma(c)\otimes q$).
        We claim that $\theta$ is an involution of $K$. Indeed, $\theta^2=\id_K$
        and for all $c,c'\in C$, $e\in E$, $q\in Q$, we have $((c\otimes q)(c'e))^\theta=
        (cc'\otimes qe)^\theta=\sigma(c)\sigma(c')\otimes qe=(\sigma(c)\otimes q)(\sigma(c')e)=(c\otimes q)^\theta(\sigma(c'e))$
        ($e=\sigma(e)$ since $e\in E=D^G$).
        This implies that $(kx)^\theta=k^\theta\sigma(x)$ for all $k\in K$, $x\in D$. Putting
        $x=1\otimes a$ and $x=a\otimes 1$ yields that $\theta$ is an involution.
    \end{proof}

    \begin{remark}\label{AZ:RM:after-Saltman-thm}
        (i) The proof of Theorem~\ref{AZ:TH:Saltman} also shows that a $C$-algebra $R$ with a Goldman
        element is Morita equivalent to its opposite over $C$ if and only if it is Morita equivalent as a $C$-algebra
        to an algebra with an involution of the first kind. We could not find a non-Azumaya algebra admitting
        a Goldman element, though.

        (ii) A slightly different proof of the ``only if'' part
        of Theorem~\ref{AZ:TH:Saltman}(ii) that does not use double progenerators appears in \cite[pp.\ 531]{Sa78}.
    \end{remark}

    In order to explain the connection of the previous proof with the Knus-Parimala-Srinivas proof of
    Saltman's Theorem (\cite[\S4]{KnParSri90}), let us backtrack our proof to see what is the algebra $B$.
    According to Theorem~\ref{AZ:TH:being-mor-eq-to-ring-with-inv}, $B=\End_A(P\oplus P^{[1]})$
    with $P$ an arbitrary $A$-progenerator, and the involution of $B$ is induced by the bilinear form $b_P$
    constructed in the proof of Theorem~\ref{AZ:TH:being-mor-eq-to-ring-with-inv}. The functor $[1]$
    is computed using a double $A$-progenerator $K$ with involution which is obtained in the proof of
    Theorem~\ref{AZ:TH:Saltman}. Let us choose $P=A_A$.
    Then $B\cong\End_A(A\oplus K_1)$ (since $A^{[1]}=\Hom_A(A_A,K_0)\cong K_1$ via $f\mapsto f(1_A)$).
    This algebra with involution is essentially the algebra with involution constructed in \cite[\S4]{KnParSri90}.
    However, in \cite{KnParSri90}, the involution of $B$ is constructed
    directly, while here we have obtained it from a general bilinear form on the
    right $A$-module $A\oplus K_1$, and hence suppressed some of the computations of \cite{KnParSri90}.

    We also note that many proofs of special cases of Saltman's Theorem in the literature,
    \cite{KnParSri90} and \cite{Sa78} in particular,
    involve the construction of an involutary map $\theta$ taking some $(A^\op,A)$-bimodule or an equivalent object $K$ into itself
    and satisfying $(a^\op kb)^\theta=b^\op k^\theta a$ ($a,b\in A$, $k\in K$).
    For example, consider
    the map $\alpha$ in \cite[pp.~532--3, 537]{Sa78}, the maps $\psi$ and $\sigma_P$ in \cite[pp.~71--2]{KnParSri90}
    and the map $u$ of \cite[pp.~196]{Jac96FinDimDivAlg}.
    This hints that these proofs can be effectively phrased using general bilinear forms.

\subsection{Rank Bounds}

    We proceed by showing that the algebra $B$ in Saltman's Theorem can be taken
    to have $\rank(B)=4\rank(A)$ in part (i) and  $\rank(B)=4\min\{\rank(A),\rank(A^\sigma)\}$ in part (ii).
    This  follows from the following technical lemmas.

    \begin{lem}\label{AZ:LM:rank-of-hom}
        Let $A$ be an Azumaya  $C$-algebra and let $M,N\in\rproj{A}$.
        Then
        \[\rank(\Hom_A(M,N))=\frac{\rank(M)\cdot\rank(N)}{\rank(A)}\ .\]
    \end{lem}

    \begin{proof}
        Let $P\in\Spec(C)$ and let $F$ be an algebraic closure of $k_P$, the fraction field of $C/P$.
        Set $A_F=A\otimes_C F$, $M_F=M\otimes_C F$ and $N_F=N\otimes_C F$.
        By Proposition~\ref{AZ:PR:standard-hom-III}, we have $\Hom_{A_F}(M_F,N_F)\cong \Hom_{A}(M, N)\otimes_CF\cong
        (\Hom_{A}(M,N)\otimes_C k_P)\otimes_{k_P}F$.
        Therefore, it is enough to verify the lemma in case $C$ is an algebraically closed field, i.e.\ when $C=F$,
        in which case $\rank_C(-)$ and $\dim_F(-)$ coincide.
        In this case, we must have $A\cong \nMat{F}{n}$ (as $F$-algebras) and $M\cong \nMat{F}{s\times n}$, $N\cong \nMat{F}{t\times n}$ (as
        $A$-modules)
        for suitable $s,t,n\in\N\cup\{0\}$. But then, $\dim_F\Hom_A(M,N)=\dim_F\nMat{F}{t\times s}=ts=\dim_F(N)\cdot \dim_F(M)/\dim_F(A)$,
        as required.
    \end{proof}

    \begin{lem}\label{AZ:LM:rank-of-double-mod}
        Let $A$ be an Azumaya $C$-algebra and let $K$ be a double $A$-progenerator of type $\sigma\in\Aut(C)$.
        Then
        \[\rank(K_0)=\sqrt{\rank(A)\rank(A^\sigma)}\quad \mathrm{and} \quad \rank(K_1)=\sqrt{\rank(A)\rank(A^{\sigma^{-1}})}\]
    \end{lem}

    \begin{proof}
        Consider $K$ as an $(A^\op,A)$-bimodule as explained in section~\ref{section:double-prog}.
        Since $A$ is of type $\sigma$, $\End_A(K_0)=\End(K_A)\cong (A^\sigma)^\op$
        as \emph{$C$-algebras}, hence $\rank(\End_A(K_0))=\rank(A^\sigma)$.
        However, by Lemma~\ref{AZ:LM:rank-of-hom}, we have $\rank(\End_A(K_0))=\rank(K_0)^2/\rank(A)$. Comparing both
        expressions yields the  formula for $\rank(K_0)$. The formula for $\rank(K_1)$ is shown in the same manner.
    \end{proof}

    \begin{prp}\label{AZ:PR:rank-bound}
        In Theorem~\ref{AZ:TH:Saltman}(i) (resp.\ Theorem~\ref{AZ:TH:Saltman}(ii)),
        the algebra $B$ can be
        chosen such that $\rank(B)=4\rank(A)$ (resp.\
        $\rank(B)=4\min\{\rank(A),\rank(A^\sigma)\}$).
    \end{prp}

    \begin{proof}
        Let $\sigma$ be $\id_C$ in the case of  Theorem~\ref{AZ:TH:Saltman}(i) and
        a nontrivial Galois automorphism of $C/C_0$ in the case of Theorem~\ref{AZ:TH:Saltman}(ii).
        In the comment after Remark~\ref{AZ:RM:after-Saltman-thm} we saw that
        $B$ can taken to be $\End_A(A\oplus K_1)$ with $K$ a double $A$-progenerator of type $\sigma$.
        In this case, by Lemmas~\ref{AZ:LM:rank-of-double-mod} and~\ref{AZ:LM:rank-of-hom},
        we have
        \[\rank(B)=\frac{\Circs{\rank(A)+\sqrt{\rank(A)\rank(A^\sigma)}}^2}{\rank(A)}=\Bigg(1+\sqrt{\frac{\rank(A^\sigma)}{\rank(A)}}\Bigg)^2\rank(A)\ .\]
        In the case of  Theorem~\ref{AZ:TH:Saltman}(i), this means $\rank(B)=4\rank(A)$, so we are done.
        For the case of Theorem~\ref{AZ:TH:Saltman}(ii), it is enough
        to show that $A$ is Brauer equivalent to an Azumaya algebra $A'$ of
        rank  $\min\{\rank(A),\rank(A^\sigma)\}$ (for then we can
        replace $A$ with $A'$ and get the desired rank).

        Recall from section~\ref{section:azumaya} that we can decompose $C$ as a product
        $C=C_1\times \dots\times C_t$ such that $\rank(A\otimes_C C_i)$ and $\rank(A^\sigma\otimes_C C_i)$ are
        constant for all $i$.
        Let $e_i\in C$ be the idempotent satisfying $C_i=e_iC$. We  identify $\Spec(C)$ with $\bigcup \Spec(C_i)$
        in the standard way.
        It is easy to see that every
        $C$-algebra $R$ factors as a product $\prod_{i}e_iR$ and $R/C$ is Azumaya if and only if
        $e_iR/e_iC$ is Azumaya for all $i$.
        Define $A'=\prod_iA'_i$ where
        $A'_i=e_iA$ if $\rank_{C_i}(e_iA)<\rank_{C_i}(e_i(A^\sigma))$ and $A'_i=e_i(A^\sigma)^\op$ otherwise.
        Then $A'\Breq A$ (since $e_iA'\Breq e_iA$ for all $i$),
        and, by definition, $\rank_C(A')=\min\{\rank_C(A),\rank_C(A^\sigma)\}$, as required.
    \end{proof}

\subsection{The Semilocal Case}
\label{subsection:semilocal-case}

        In \cite[\S4]{Sa78}, Saltman also showed that when $C$ is semilocal and connected, one
        can take $B=A$ in Theorem~\ref{AZ:TH:Saltman} (Albert's Theorems follow as the special case
        where $C$ is a field).
        Let us recover the proof using the language and the tools of the previous sections:

        Since $C$ is connected, $A$ has constant rank,
        say $\rank(A)=n^2$.
        Suppose the assumptions of parts (i) or (ii)
        of Theorem~\ref{AZ:TH:Saltman} are satisfied, and set $\sigma=\id_C$ for part (i).
        By the proof of Theorem~\ref{AZ:TH:Saltman}, there is
        a double $A$-module $K$ of type $\sigma$ admitting an involution $\theta$. We claim that
        $K$ is the standard double $A$-module of some anti-automorphism $\gamma:A\to A$ (see section~\ref{section:transfer}).
        Indeed, by assumption, $\rank(A)=n^2=\rank(A^\sigma)$, hence by Lemma~\ref{AZ:LM:rank-of-double-mod},
        $\rank(K_1)=\rank(A)$. We now claim that for $P,Q\in \rproj{A}$, $\rank(P)=\rank(Q)$ implies $P\cong Q$.
        By tensoring with $C/\Jac(C)$, we may assume $C$ is a finite
        product of fields and $A$ is a finite product of central simple algebras, in which case the
        claim is routine. Therefore, $K_1\cong A_A$, and arguing as in the proof of Theorem~\ref{AZ:TH:anti-autos-are-transposes}
        yields the desired $\gamma$.

        We can now apply the proof of Theorem~\ref{AZ:TH:involution-transfer}
        with $R=A$ (which is semilocal) to assert that $A$ has an involution of type $\sigma$.
        However, we need to  verify the assumptions of the theorem, so let $n_i$, $D_i$ and $F_i$
        be defined as in Theorem~\ref{AZ:TH:involution-transfer} for $R=A$.
        Since $\rank(A)=n^2$, we have
        $\dim_{F_i}\nMat{D_i}{n_i}=n^2$ for all $i$.
        Thus, when $n$ is even, for all $i$, either $n_i$ is even or $D_i$ is a division ring, as required.
        When $n$ is odd and $\sigma\neq\id_C$, set
        $\quo{A}=A/\Jac(A)$, $\quo{C}=C/\Jac(C)$ and $\quo{C}_0=C_0/\Jac(C_0)$. It is a standard fact that
        $\quo{A}/\quo{C}$ is Azumaya (so $\quo{C}=\Cent(\quo{A})$) and $\quo{C}/\quo{C}_0$ is Galois with
        Galois group $\{\quo{\sigma},\id\}$, where $\quo{\sigma}$ is the map induced by $\sigma$ on $\quo{C}$.
        Noting that $\quo{C}=F_1\times\dots\times F_t$ and $\rank_{\quo{C}_0}(\quo{C})=2$, we see that $\quo{\sigma}$ cannot fix
        any $F_i$ pointwise, so again the assumptions of Theorem~\ref{AZ:TH:involution-transfer} hold.
        Finally, in case $n$ is odd and $\sigma=\id_C$, we have $A^{\otimes 2}\Breq C$ (by assumption)
        and $A^{\otimes n}\Breq C$
        (by Theorem~\ref{AZ:TH:exponent-rank-relation}, since $\rank(A)=n^2$),
        which implies $A\Breq C$, i.e.\ $A\cong \End_C(P)$ for some $C$-progenerator $P$, necessarily
        of rank $n$. By what we have shown above, this implies $P\cong C^n$, so $A\cong\End_C(P)\cong\nMat{C}{n}$
        and $A$ has the transpose involution.

        We note that specializing in Theorem~\ref{AZ:TH:involution-transfer} to the case $R$ is an Azumaya $C$-algebra
        and $C$ is connected
        saves many technicalities in the last steps of the proof. In particular, one can use
        the Skolem-Noether Theorem rather than basing on Theorem~\ref{AZ:TH:anti-autos-are-transposes}; cf.\
        the argument given by Saltman in \cite[pp.\ 538]{Sa78}.

    \begin{remark}\label{AZ:RM:semilocal-Az-algs}
        Assume $C$ is semilocal or $\Q$-finite and let $\tau$ be \emph{any} automorphism of $C$.
        Then Corollary~\ref{AZ:CR:three-implies-two} (together with Theorem~\ref{AZ:TH:Bass-criterion})
        implies that $A$ is Brauer equivalent to an Azumaya algebra with an anti-automorphism of type $\tau$
        if and only if $A^\op\Breq A^\tau$. The semilocal part of this claim is somewhat standard since it is well-known
        that
        $A^\op\Breq A^\tau$ implies $A^\op\cong A^\tau$ in this case (e.g.\ see the previous paragraphs).
        However, the $\Q$-finite case seems to be unknown.
        When $\tau$ is a Galois automorphism of order $2$ (i.e.\ $C/C^{\{\tau\}}$ is
        Galois with Galois group $\{1,\tau\}$), the condition $A^\op\Breq A^\tau$ is equivalent to $A\otimes_C A^\tau
        \cong \Res_{C/C_0}(\Cor_{C/C_0}(A))\sim C$. In this special case, it could be that the claim is true
        without any assumption on $C$.
    \end{remark}

\section{Azumaya Algebras of Exponent $2$ Without Involution}
\label{section:no-inv}

    Let $C$ be a commutative ring and let $A$ be an Azumaya $C$-algebra with $A\Breq A^\op$.
    Then Saltman's Theorem asserts that there is some algebra $B\Breq A$ with involution of the first kind.
    However,
    it can be very tricky to construct examples of algebras $A$ with $A\Breq A^\op$ that do not have an involution already.
    For example, Saltman have shown in \cite[\S4]{Sa78} that this is impossible if $C$ is semilocal or $\rank(A)=4$.
    In particular, $A$ cannot be  a tensor product of quaternion algebras.
    (The proof of the semilocal case was recovered in subsection~\ref{subsection:semilocal-case}. In case $\rank(A)=4$, the map
    $a\mapsto \mathrm{Tr}(a)-a$ is an involution on $A$, where $\mathrm{Tr}(a)$ denotes the \emph{reduced trace} of $a$.)
    Saltman has pointed out in \cite[\S4]{Sa78} how to produce examples of Azumaya algebras without involution that are trivial in $\Br(C)$.
    However, we have not seen in the literature  examples of non-trivial such algebras.
    Of course, one can always take the product of a non-trivial
    quaternion algebra with a trivial Azumaya algebra without involution, but this results in $C$ being non-connected
    (i.e.\
    $C$ is a product of two rings). One is therefore interested in examples in which $C$ is connected. More generally,
    one can ask whether the bound $\rank(B)\leq 4\rank(A)$ of Proposition~\ref{AZ:PR:rank-bound} is tight.

    In this section, we shall use the results of the section~\ref{section:double-prog} to explain
    how to detect whether an Azumaya algebra admits an  anti-automorphisms. This will be used to construct
    an explicit example of a non-trivial Azumaya algebra over a Dedekind domain  $C$  that does not
    have an involution.
    However, this example does not show that the bound $\rank(B)\leq 4\rank(A)$ is optimal,
    and it also
    has zero divisors. (Even worse, we shall see that for $C$ as in the
    example, and also for many other low-dimensional domains,  we can always choose $B$ to be without
    zero divisors and to satisfy  $\rank(B)\leq\rank(A)$.)

    We thank Asher Auel for pointing us to this problem.

    \medskip

    Throughout, $C$ is a commutative ring and $\sigma\in\Aut(C)$. Unless specified otherwise, all tensor products
    are taken over $C$.
    Recall that the \emph{Picard group}, $\Pic(C)$, consists of the isomorphism classes of rank-$1$ projective
    $C$-modules, multiplication being  tensor product (over $C$). Equivalently, $\Pic(C)$ consists of isomorphism classes
    of $(C,C)$-progenerators
    of type $\id_C$. In case $C$ is a Dedekind domain (e.g.\ a number ring), $\Pic(C)$ is just the \emph{ideal class group} of $C$
    (e.g.\ see \cite{MaximalOrders}).

    \begin{thm}\label{AZ:TH:auti-auto-test}
        Let $A$ be an Azumaya $C$-algebra such that $A^\op\Breq A^\sigma$,  let $K$ be
        any $(A^\op,A^\sigma)$-progenerator of type $\id_C$ (cf.\ Theorem~\ref{AZ:TH:Bass-criterion}),
        which we also view as a double $A$-progenerator (of type $\sigma$),
        let $P$ be a right $A$-progenerator and set $P^{[1]}=\Hom_A(P,K_0)$.
        Then the following conditions are equivalent:
        \begin{enumerate}
            \item[(a)] $B:=\End_A(P)$ has an anti-automorphism of type $\sigma$.
            \item[(b)] There is $[I]\in\Pic(C)$ such that $I\otimes P\cong  P^{[1]}$ as $A$-modules.
            \item[(c)] $\End_A(P)\cong\End_A(P^{[1]})$ as $C$-algebras.
        \end{enumerate}
        In particular, if $I\otimes P\ncong P^{[1]}$ for all $[I]\in \Pic(C)$, then $\End_A(P)$
        does not have an anti-automorphism of type $\sigma$.
    \end{thm}

    \begin{proof}
        (a)$\derives$(b): By Proposition~\ref{AZ:PR:existence-of-anti-auto-crit}, there exists a double $A$-progenerator
        $K'$ of type $\sigma$ and a regular bilinear form
        $b:P\times P\to K'$. View $K'$ as a right $A\otimes A^\sigma$-module via $k(a\otimes a')=k\mul{0}a\mul{1}a'$
        and view $K$ as a right $A\otimes A^\sigma$ in the standard way. By Proposition~\ref{AZ:PR:progenerators-between-Azu-algs},
        both $K$ and $K'$ are $(C,A\otimes A^\sigma)$ progenerators. Thus, there exists a $(C,C)$-progenerator
        $I$ such that $K'=I\otimes_C K$. As both $K$ and $K'$ have type $\sigma\in\Aut(C)$, $I$ is of type $\id_C$
        and hence $[I]\in\Pic(C)$. Now, by
        Proposition~\ref{AZ:PR:hom-of-tensor}, we have
        $P\cong \Hom(P,K'_0)\cong \Hom_{C\otimes A}(C\otimes P,I\otimes K_0)=\Hom_C(C,I)\otimes\Hom_A(P,K_0)\cong
        I\otimes P^{[1]}$.

        (b)$\derives$(c): We have $\End_C(I)=C$ since $I$ is a $(C,C)$-progenerator of type $\id_C$, so
        by Proposition~\ref{AZ:PR:hom-of-tensor}, $\End_A(P^{[1]})\cong\End_{C\otimes A}(I\otimes P,I\otimes P)\cong
        \End_C(I)\otimes \End_A(P)\cong C\otimes \End_A(P)\cong \End_A(P)$.
        The isomorphism is easily seen to be a $C$-algebra isomorphism.

        (c)$\derives$(a): Let $\psi:\End_A(P^{[1]})\to\End_A(P)$ be an isomorphism.
        The map $w\mapsto w^{[1]}:\End_A(P)\to \End_A(P^{[1]})$ is an anti-\emph{homo}morphism of rings (since
        $[1]$ is an additive contravariant functor).
        Moreover, it is an anti-\emph{iso}morphism of rings since $[1]:\rproj{A}\to\rproj{A}$ is a duality
        of categories (Lemma~\ref{AZ:LM:right-reg-iff-left-reg}(ii)). It is also straightforward to check that
        for $c\in C\subseteq \End_A(P)$, one has $c^{[1]}=\sigma(c)$.
        Thus, $\alpha:\End_A(P)\to\End_A(P)$ defined by $w^\alpha=\psi(w^{[1]})$ is an anti-automorphism
        of type $\sigma$.
    \end{proof}

    \begin{remark}
        By applying Theorem~\ref{AZ:TH:auti-auto-test} with $P=A_A$, we can test whether $A$ itself has
        an involution. In this case, $P^{[1]}$ is just $K_1$ (via $f\mapsto f(1)$), so we only need
        to verify that $I\otimes A\ncong K_1$ as $A$-modules for all $[I]\in \Pic(C)$.
    \end{remark}

    We shall now use Theorem~\ref{AZ:TH:auti-auto-test} to construct an Azumaya algebra
    $B$ over a Dedekind domain $C$ such that $B\otimes B\Breq C$, $B\nBreq C$ and $B$ does not
    have a  $C$-anti-automorphism. For a fractional ideal $I$ of $C$, we shall
    write $I^{\otimes n}$ for the $n$-th power of $I$ as a fraction ideal,
    and $I^{\oplus n}$ for the $C$-module $I\oplus \dots\oplus I$ ($n$ times). The ambiguous notation $I^n$ shall
    be avoided henceforth, except for writing $I^{-1}$ to denote the inverse of $I$.
    We shall make extensive usage of some well-known facts about projective modules over  Dedekind domain, such as:
    Every f.g.\ projective module over a Dedekind domain $C$ is a direct sum of fractional ideals, and for
    fractional ideals $I_1,\dots,I_n,J_1,\dots,J_n$,
    we have $\bigoplus_{i=1}^n I_i\cong \bigoplus_{i=1}^n J_i$ if and only if
    $[\prod_i I_i]=[\prod_i J_i]$ in $\Pic(C)$. We refer the reader to \cite{MaximalOrders} for proofs.

    \begin{example}\label{GEN:EX:last-example}
        Let $C$ is a Dedekind domain admitting a \emph{non-trivial} quaternion Azumaya $C$-algebra $A$.
        By \cite[Th.~4.1]{Sa78}, $A$ has an involution
        of the first kind $\alpha$, which gives rise to a double $A$-module $K$ with $K_1\cong A_A$ (see section~\ref{section:transfer}).
        Viewing $A$ as a $C$-module, we can write $A_C\cong C^{\oplus3}\oplus J$  with $J$ a fractional ideal.

        Suppose further that there is $[L]\in\Pic(C)$
        such that $[L^{\otimes 8}]$ is not a $16$-th power in $\Pic(C)$.
        We claim that the Azumaya $C$-algebra
        \[B:=\End_C(C^{\oplus 3}\oplus L) \otimes A ,\]
        does not have an anti-automorphism of type $\id_C$,
        despite clearly satisfying $B\otimes B\Breq C$ and $B\nBreq C$.
        We show this by applying Theorem~\ref{AZ:TH:auti-auto-test} with $P= (C^{\oplus 3}\oplus L) \otimes A_A$.
        (The module $P$
        is a right $C\otimes A$-module which we view as a right $A$-module in the standard way. We have
        $\End_A(P)\cong B$ by Proposition~\ref{AZ:PR:hom-of-tensor}.)
        In fact, we shall show that $I\otimes P\ncong P^{[1]}$ as \emph{$C$-modules} for all $[I]\in\Pic(C)$.

        Indeed, observe that $P^{[1]}=\Hom_A(P,K_0)=\Hom_{C\otimes A}((C^{\oplus3}\oplus L)\otimes A_A, C\otimes K_0)\cong
        \Hom_C(C^{\oplus3}\oplus L,C)\otimes \Hom_A(A_A,K_0)\cong(C^{\oplus3}\oplus L^{-1})\otimes K_1$.
        Thus, we have
        \begin{eqnarray*}
        P^{[1]}&\cong& (C^{\oplus 3}\oplus L^{-1})\otimes K_1\cong (C^{\oplus3}\oplus L^{-1})\otimes (C^{\oplus 3}\oplus J)\\
        &\cong &
        C^{\oplus 9}\oplus (L^{-1})^{\oplus 3}\oplus J^{\oplus 3}\oplus L^{-1}J\cong C^{15}\oplus (L^{-1})^{\otimes 4}J^{\otimes 4}\ .
        \end{eqnarray*}
        (as $C$-modules).
        On the other hand, for $[I]\in \Pic(C)$,
        \begin{eqnarray*}
        I\otimes P &=& I\otimes (C^{\oplus 3}\oplus L)\otimes A\cong
        I\otimes (C^{\oplus 3}\oplus L)\otimes (C^{\oplus 3}\oplus J)\\
        &\cong& I^{\oplus 9}\oplus (IJ)^{\oplus3}\oplus(IL)^{\oplus3}\oplus
        (ILJ)\cong C^{\oplus15}\oplus I^{\otimes 16}L^{\otimes 4}J^{\otimes 4}\ .
        \end{eqnarray*}
        It follows that $I\otimes P\cong P^{[1]}$ as $C$-modules
        if and only if $[I^{\otimes 16}L^{\otimes 4}J^{\otimes 4}]=[(L^{-1})^{\otimes 4}J^{\otimes 4}]$ in $\Pic(C)$, or equivalently, if
        $[(I^{-1})^{\otimes 16}]=[L^{\otimes 8}]$. But this is impossible by the assumption on $L$.

        We remark that if $L$ is chosen such that $[L^{\otimes 16}] =[C]$, then
        $(C^{\oplus3}\oplus L)^{\oplus 16}\cong C^{\oplus 64}$. In this case, $\nMat{B}{16}\cong A\otimes\nMat{C}{16}$, so $\nMat{B}{16}$
        has an involution of type $\id_C$. As a result,
        one of the algebras $B$, $\nMat{B}{2}$, $\nMat{B}{4}$, $\nMat{B}{8}$
        does not an involution of type $\id_C$ while its $2\times 2$ matrix algebra does  have such
        involution. In addition, if $\Aut(C)=\id_C$ (e.g.\ if the fraction field of $C$ has no nontrivial automorphisms),
        then $B$ has no anti-automorphisms at all.

        Explicit choices of $C$, $A$ and $L$ satisfying all previous conditions (including $[L^{\otimes 16}] =[C]$ and $\Aut(C)=\id_C$)
        are the following: Let
        $D$ be the integer ring of $K:=\Q[x \where x^3+x+521=0]$, and
        take
        \[C=D[2^{-1}],\qquad A=C[i,j\where i^2=j^2=-1,~ij=-ji],\qquad L=\ideal{113,{x}-16}^{\otimes 3}\ .\]
        That $A$ is Azumaya follows from the general fact (left to the reader) that
        $(a,b)_{2,C}:=C[i,j\where i^2=a,~j^2=b,~ij=-ji]$ is Azumaya for all $a,b\in \units{C}$, provided $\frac{1}{2}\in C$.
        The algebra $A$ is non-trivial in $\Br(C)$ since $K$ (and hence $C$) embeds in $\R$, and $A\otimes_C \R\cong (-1,-1)_{2,\R}$ is
        well-known to be a division ring.
        The class group of $D$ is a cyclic group of
        order $48$ generated by $[\ideal{113,{x}-16}]$ (verified using SAGE), so when viewed as an
        element of $\Pic(D)$, $[\ideal{113,{x}-16}^{\otimes 3}]$
        has order $16$ and its eighth power is not a $16$-th power in $\Pic(D)$. Adjoining $\frac{1}{2}$ to $D$
        does not affect these facts by the following easy lemma (note that $2$ is prime in $D$).
    \end{example}

    \begin{lem}
        Let $D$ be a Dedekind domain, let $0\neq p\in D$ be a prime \emph{element} of $D$, and set $D_p=D[\frac{1}{p}]$.
        Then the  map $[I]\mapsto [ID_p]: \Pic(D)\to \Pic(D_p)$ (known
        as the restriction map) is an isomorphism
        of groups.
    \end{lem}

    \begin{proof}
        That $[I]\mapsto [ID_p]$ is a group homomorphism is straightforward.
        It is onto since any ideal $I\idealof D_p$ is easily seen to satisfy
        $(I\cap D)D_p=I$.

        To show injectivity, observe that  $D$ has Krull-dimension $1$, hence $pD$ is maximal.
        As a result,  every ideal $0
        \neq J\idealof D$ can be written as $J=p^{n}J'$ ($n\in\N$) with
        $J'$ is coprime to $p$.
        The same applies to elements $x\in D$.
        Also observe that if $J$ is coprime to $p$
        and $x\in D$, the $px\in J$ implies $x\in J$. Indeed, $px\in J\cap pD=pJ$ (since
        $J$ is coprime to $p$), hence $x\in J$.

        Suppose now that $J\idealof D$ is such that
        $I:=JD_p$ is principal. We need to show that $J$ is principal
        (i.e.\ that $[J]$ is trivial in $\Pic(D)$).
        Replacing $J$ with $p^{-n}J$ if needed, we may assume
        that $J$ is coprime to $p$. Let $z$ be a generator of $I$. Again, replacing $z$ with $p^nz$ ($n\in\Z$),
        we may assume that $z\in J$, and since $J$ is coprime to $p$, we may further divide $z$ by $p$ unitl it
        is coprime to $p$. Let $y\in J$.
        Then $y=z(p^nw)$ for some $n\in\Z$ and $w\in D$ coprime to $p$.
        This implies $p^{-n}y=zw$, so $n$ cannot be negative (since the right hand side is not divisible by $p$).
        Therefore, $wp^n\in D$ and $y=zwp^n\in zD$. This means that $J=zD$,
        as required.
    \end{proof}

    \begin{remark}\label{AZ:RM:non-optimal-bound}
        Let $C$ be a regular integral domain of Krull dimension at most $2$, and let $F$ be the fraction
        field of $C$. Assume further that
        every central simple $F$-algebra of exponent $2$ is Brauer equivalent to a quaternion algebra.
        For example, this is the case when $F$ is a global field by the Albert-Brauer-Hasse-Noether Theorem,
        or when $F$ has transcendence degree at most $2$ over a separably closed subfield by \cite{deJong04}.
        Then every Azumaya  $C$-algebra $A$ of exponent $2$ admits an algebra $B\in [A]$ with
        involution and $\rank_C(B)\leq \rank_C(A)$. In particular, in this case,
        the bound $\rank_C(B)\leq 4\rank_C(A)$ of Proposition~\ref{AZ:PR:rank-bound} is not tight.
        Indeed, by Saltman \cite[Th.~4.1]{Sa78}, it is enough to show that $B$ can be chosen to be a quaternion algebra.
        As $A':=A\otimes_C F$ has exponent $2$, there exists by assumption a \emph{quaternion} central simple $F$-algebra $B'\Breq A'$.
        Now, by the proof of \cite[Pr.\ 7.4]{AusGold60}, $B'$ contains an Azumaya $C$-algebra $B$ such that $B'=B\otimes_C F$.
        Finally, by  \cite[Th.\ 7.2]{AusGold60}, $B\Breq A$, so we are done.
    \end{remark}

\section{Questions}
\label{section:questions}

    We finish with several questions that we were unable to answer.

    \begin{que}
        Is there a ring $R$ such that $R\Moreq R^\op$ but $R$ is not Morita equivalent
        to a ring with an anti-automorphism (cf.\ section~\ref{section:overview})?
    \end{que}

    \begin{que}
        Is there a semilocal  ring $R$ such that  $R$ does not have an involution
        despite $\nMat{R}{2}$ having an involution (cf.\ Theorem~\ref{AZ:TH:involution-transfer})?
    \end{que}

    \begin{que}
        Is there a non-Azumaya algebra admitting a Goldman element (cf.\ Remark~\ref{AZ:RM:after-Saltman-thm}(i))?
    \end{que}

    \begin{que}
        Let $C/C_0$ be a Galois extension with Galois group $\{1,\sigma\}$.
        Is it true that all Azumaya $C$-algebras $A$ with $A\Breq A^\sigma$
        are Brauer equivalent to an Azumaya algebra with an anti-automorphism
        whose restriction to $C$ is $\sigma$  (cf.\ Remark~\ref{AZ:RM:semilocal-Az-algs})?
    \end{que}

    \begin{que}
        Is the bound $\rank(B)\leq 4\rank(A)$ of Proposition~\ref{AZ:PR:rank-bound}
        tight when the base ring is connected (cf.\ Remark~\ref{AZ:RM:non-optimal-bound})?
    \end{que}

    \begin{que}
        Are there Azumaya algebras of exponent $2$ and
        without zero-divisors that do not admit an involution of the first kind?
    \end{que}

    For our last question, recall that an involution $\alpha$ on a ring
    $R$ is called \emph{hyperbolic} if there exists an idempotent $e\in R$ such that $e+e^\alpha=1$.
    In all  proofs of Saltman's Theorem that we have encountered (e.g.\ \cite{Sa78}, \cite{KnParSri90}, and the proof
    given in section~\ref{section:first-kind}),
    the constructed involution was hyperbolic. In contrast, Albert's Theorem guarantees an involution which is non-hyperbolic
    (and in fact \emph{anisotropic}); indeed, we can take $B$ to be a division ring, and any involution on a division ring is non-hyperbolic.
    We therefore ask:

    \begin{que}
        Let $A$ be an Azumaya $C$-algebra with $A\otimes A\Breq C$. Does there always exist
        an Azumaya algebra $B\in[A]$ having a \emph{non-hyperbolic} involution of the first kind?
    \end{que}

\bibliographystyle{plain}
\bibliography{MyBib}

\end{document}